\documentclass[a4paper,12pt]{article}
\usepackage{amsthm}
\usepackage{amsmath}
\usepackage{latexsym}
\usepackage{url}
\usepackage{color}
\usepackage{comment}
\numberwithin{equation}{section}

\def\R{{\bf R}}

\def\N{{\bf N}}

\def\d{\displaystyle}
\def\e{{\varepsilon}}

\def\p{\partial}

\newcommand{\dt}{\partial_{t}}%

\newcommand{\ga}{\gamma}

\newcommand{\ka}{\kappa}
\newcommand{\si}{\sigma}
\newcommand{\la}{\lambda}

\newcommand{\supp}{\operatorname{supp}}

\newcommand{\I}{\infty}

\newcommand{\EQS}[1]{\begin{align} #1 \end{align}}
\newcommand{\EQQS}[1]{\begin{align*} #1 \end{align*}}
\newcommand{\EQQ}[1]{\begin{equation*} \begin{split} #1
 \end{split} \end{equation*}}

\newcommand{\LR}[1]{{\langle {#1} \rangle }}

\newtheorem{thm}{Theorem}[section]
\newtheorem{cor}{Corollary}[section]
\newtheorem{lem}{Lemma}[section]
\newtheorem{prop}{Proposition}[section]
\newtheorem{rem}{Remark}[section]
\newtheorem{definition}{Definition}[section]

\title{Critical exponent for the wave equation with a time-dependent scale invariant damping and a cubic convolution}
\author{
Masahiro Ikeda
\footnote{Department of Mathematics, Faculty of Science and Technology, Keio University, 3-14-1 Hiyoshi,
Kohoku-ku, Yokohama, 223-8522, Japan/Center for Advanced Intelligence Project, RIKEN, Japan. email: masahiro.ikeda@keio.jp/masahiro.ikeda@riken.jp.}
,
Tomoyuki Tanaka
\footnote{Graduate School of Mathematics, Nagoya University, Chikusa-ku, Nagoya, 464-8602, Japan/Center for Advanced Intelligence Project, RIKEN, Japan. e-mail: d18003s@math.nagoya-u.ac.jp/tomoyuki.tanaka.hw@riken.jp}
\ and \
Kyouhei Wakasa
\footnote{Department of Creative Engineering, National Institute of Technology, Kushiro College, 2-32-1 Otanoshike-Nishi, Kushiro-Shi, Hokkaido 084-0916, Japan. e-mail: wakasa@kushiro-ct.ac.jp.
}
}
\date{
\[
\begin{array}{llll}
\mbox{\footnotesize{\bf Keywords:}}
& \mbox{\footnotesize Wave equation, Time-dependent scale invariant damping,}\\
& \mbox{\footnotesize Cubic convolution, Global existence, Blow-up,}\\
& \mbox{\footnotesize Critical exponent, Lifespan, Dissipation}\\
\mbox{\footnotesize{\bf MSC2010:}}
& \mbox{\footnotesize Primary 35B33 Secondary 35B44, 35L05, 35L71, 35B45}\\
\end{array}
\]
}
\begin{document}
\maketitle
\begin{abstract}
In the present paper, we study the Cauchy problem for the wave equation with a time-dependent scale invariant damping $\frac{2}{1+t}\dt v$ and a cubic convolution $(|x|^{-\ga}*v^2)v$ with $\ga\in \left(-\frac{1}{2},3\right)$ in three spatial dimension for initial data $\left(v(x,0),\partial_tv(x,0)\right)\in C^2(\R^3)\times C^1(\R^3)$ with a compact support, where $v=v(x,t)$ is an unknown function to the problem on $\R^3\times[0,T)$. Here $T$ denotes a maximal existence time of $v$.

The first aim of the present paper is to prove unique global existence of the solution to the problem and asymptotic behavior of the solution in the supercritical case $\gamma\in (0,3)$, and show a lower estimate of the lifespan in the critical or subcritical case $\gamma\in \left(-\frac{1}{2},0\right]$. The essential part for their proofs is to derive a weaker estimate under the weaker condition (Proposition \ref{lm:apriori1} and Theorem \ref{lem:potential}) than the case without damping (see \cite{T03}) and to recover the weakness by the effect of the dissipative term.

The second aim of the present paper is to prove a small data blow-up and the almost sharp upper estimate of the lifespan for positive data with a compact support in the subcritical case $\gamma\in \left(-\frac{1}{2},0\right)$. The essential part for the proof is to refine the argument for the proof of Theorem 6.1 in \cite{H20} to obtain the upper estimate of the lifespan.

Our two results determine that a critical exponent $\gamma_c$ which divides global existence and blow-up for small solutions is $0$, namely $\gamma_c=0$. As the result, we can see that the critical exponent shift from $2$ to $0$ due to the effect of the scale invariant damping term.
\end{abstract}

\tableofcontents


\section{Setting of our problem and known results}

\ \ \ \
In the present paper, we study the Cauchy problem
for the wave equation with a time-dependent scale invariant damping and a cubic convolution in three spatial dimension:
\EQS{\label{IVP}
\begin{cases}
  \dt^2 v-\Delta v+\displaystyle{\frac{\mu}{1+t}}\dt v=(V_{\gamma}*|v|^2)v, &(x,t)\in\R^3\times [0,T),\\
  v(x,0)=v_0(x), &x\in\R^3,\\
  \dt v(x,0)=v_1(x), &x\in\R^3.
\end{cases}
}
Here $T=T(v_0,v_1)\ge 0$ denotes the maximal existence time of the function $v$, which is called lifespan (see Definition \ref{def1}), $V_{\gamma}(x):=|x|^{-\gamma}$ is a given function on $\R^3$ and is called the power potential, where $\gamma\in \left(-\frac{1}{2},3\right)$ is a constant independent of $x,t$, $*$ stands for the convolution in the space variables, $\mu\ge 0$ is a non-negative constant, $v=v(x,t)\in C(\R^3\times[0,T))$ is an unknown function on $\R^3\times [0,T)$, $(v_0,v_1)\in C^2(\R^3)\times C^1(\R^3)$ is a given $\R^2$-valued function on $\R^3$ with a compact support.

In the physical context, the stationary problem corresponding to \eqref{IVP} with a mass term and the Coulomb potential, i.e. $\ga=1$
\[
    -\Delta v+v=(|x|^{-1}*|v|^2)v,\ \ \ x\in \R^n
\]
was proposed by Hartree as a model for the helium atom, where $n\ge 3$.

Menzala and Strauss \cite{MS82} studied the Cauchy problem \eqref{IVP} with a potential $V\in L^{\frac{n}{3}}(\R^n)+L^{\infty}(\R^n)$ instead of the power potential $|x|^{-\gamma}$ and without the dissipative term ($\mu=0$) in three or higher spatial dimension $n\ge 3$. They proved a large data local well-posedness result and a small data scattering result for the above $V$ in the energy space $H^1(\R^n)\times L^2(\R^n)$, where $H^1(\R^n)$ denotes the usual first order $L^2$-based Sobolev space. We note that the power potential $|x|^{-\gamma}$ with $\gamma<0$ does not belong to $V^{\frac{n}{3}}(\R^n)+L^{\infty}(\R^n)$.

The first equation of \eqref{IVP} is invariant under the scale transformation $v\mapsto v_{\sigma}$ for $\sigma>0$ given by
\begin{equation}
\label{scale}
  v_{\sigma}(x,t):=\si^{\frac{5-\gamma}{2}}v\left(\si x,\si(1+t)-1\right)
\end{equation}
on $\R^3\times [0,(T+1)/\sigma-1)$. Therefore the damping term $\frac{\mu}{1+t}\partial_tu$ is called the scale invariant damping term and is known as a threshold between“wave-like” region and “heat-like” region (see \cite{Wi06} for linear problems).

In this paper we consider the case of the three spatial dimension $n=3$ and a specific coefficient $\mu=2$.

The first aim is to prove small data unique global existence for the problem (\ref{IVP}) in the supercritical case $\gamma\in (0,3)$. Moreover, we prove the sharp space-time decay estimate of the parabolic type of the global solution $v$ as $t\rightarrow\infty$ in the supercritical case $\gamma\in (0,3)$. On the other hand, we show a lower estimate of the lifespan $T$ for small solutions in the critical or subcritical case $\gamma\in \left(-\frac{1}{2},0\right]$. The essential part for their proofs is to derive a weaker estimate under the weaker condition (see Proposition \ref{lm:apriori1} and Theorem \ref{lem:potential} for more precise) than the case without damping (see \cite{T03}) and to recover the weakness by the effect of the dissipative term $\frac{2}{1+t}\partial_tv$.

The second aim of the present paper is to prove a small data blow-up result and the almost sharp upper estimate of the lifespan $T$ for small solutions for suitable non-negative data with a compact support in the subcritical case $\gamma\in \left(-\frac{1}{2},0\right)$. The essential part for the proof is to estimate the integral of the solution with respect to the space variable more delicately than the proof of Theorem 6.1 in \cite{H20} in order to get an upper estimate of the lifespan (see Section \ref{blow-up}).

Our two results determine that a critical exponent $\gamma_c$ to the problem (\ref{IVP}) with $\mu=2$ which divides global existence and blow-up for small solution is $0$, namely $\gamma_c=0$. As the result, we can see that the critical exponent shift from $2$ to $0$ due to the effect of the scale invariant damping term ($\frac{\mu}{1+t}\partial_t v$) compared to the case without the damping ($\mu=0$).

In the undamped $(\mu=0)$ and the power type nonlinearity $|v|^p$ with $p>1$ case, i.e.,
\EQS{\label{eq_wave1}
  \begin{cases}
    \dt^2 v-\Delta v= |v|^p, &(x,t)\in\R^n\times (-T,T),\\
    v(x,0)=v_0(x), &x\in\R^n,\\
    \dt v(x,0)= v_1(x), &x\in\R^n,
  \end{cases}
}
determining a critical exponent which divides global existence and blow-up for small solutions has been extensively studied by many authors (see \cite{YZ06, Z07} and their references). It is well known that the critical exponent $p_c$ for rapidly decaying solutions as $|x|\rightarrow\infty$ to the problem \eqref{eq_wave1} is the Strauss exponent $p_0$, namely $p_c=p_0$, where the Strauss exponent $p_0$ is given by
\begin{align}
\label{strauss}
	p_0=p_{0}(n):=
	\left\{
	\begin{array}{ll}
		\infty, & (n=1),\\
		\d\frac{n+1+\sqrt{n^2+10n-7}}{2(n-1)}, & (n \geq 2).
	\end{array}
	\right.
\end{align}
More precisely, small data global existence to the problem (\ref{eq_wave1}) for smooth sufficiently rapidly decaying data $(v_0,v_1)$ as $|x|\rightarrow\infty$ holds if $p>p_0(n)$, on the other hand, small data blow-up occurs for suitable non-negative data $(v_0,v_1)$ decaying rapidly as $|x|\rightarrow\infty$ if $1<p\le p_0(n)$ (see \cite{ISW19} for recent progress).

Next we recall several results for the undamped $(\mu=0)$ and a cubic convolution
case, i.e.,
\EQS{\label{eq_wave2}
  \begin{cases}
    \dt^2 v-\Delta v=(V_{\zeta}*v^2)v,&(x,t)\in\R^n\times (-T,T),\\
    v(x,0)=v_0(x), &x\in\R^n,\\
    \dt v(x,0)=v_1(x), &x\in\R^n,
  \end{cases}
}
where $n\in \N$ with $n\ge2$, $V_{\zeta}(x):=|x|^{-\zeta}$ is the inverse power potential with $\zeta\in (0,n)$ and
$*$ stands for the convolution in the space variables.

Hidano \cite{H20} proved scattering results (Theorem 1 and Theorem 2 in \cite{H20}) related to the problem (\ref{eq_wave2}) with $\zeta\in \left(2,\frac{5}{2}\right)$ in three dimensional case $n=3$ for small data decaying rapidly as $|x|\rightarrow \infty$. The proofs of them are based on the so-called vector field method introduced in \cite{K85}. He also proved a small data blow-up result (Theorem 6.1 in \cite{H20}) to (\ref{eq_wave2}) with $\zeta\in (0,2)$ for some positive initial data
with a compact support. From his results, we see that the critical exponent $\zeta_c$ which divides global existence and blow-up to the problem (\ref{eq_wave2}) with compactly supported data is $2$, that is, $\zeta_c=2$.

Tsutaya \cite{T03} studied the Cauchy problem (\ref{eq_wave2}) with initial data $(v_0,v_1)$ decaying slowly as $|x|\rightarrow\infty$. Here slowly decaying data means a pair of given functions $(v_0,v_1)$ satisfying the following spatial decay condition:
\begin{equation}
\label{slow-decay-1}
    v_0(x)=O(|x|^{-\nu}),\ \ \ v_1(x)=O\left(|x|^{-(1+\nu)}\right)
\end{equation}
as $|x|\rightarrow\infty$ with some $\nu>0$.
Tsutaya proved a small data global existence result (Theorem 2.1 in \cite{T03}) to the problem (\ref{eq_wave2}) with $\zeta\in (2,3)$ for the initial data $(v_0,v_1)$ satisfying (\ref{slow-decay-1}) with a scaling subcritical exponent $\nu>\frac{5-\zeta}{2}$. On the other hand, he proved a
small data blow-up result (Theorem 3.4 in \cite{T03}) to the problem (\ref{eq_wave2}) with $\zeta\in (0,3)$ for the data $(v_0,v_1)$ satisfying (\ref{slow-decay-1}) with a scaling supercritical $\nu\in (\frac{1}{2}, \frac{5-\gamma}{2})$ and the more precise estimates
\begin{equation}
\label{slow-decay-2}
   v_0(x)\equiv 0,\ \ \ v_1(x)\ge \frac{\epsilon}{(1+|x|)^{1+\nu}}
\end{equation}
on $\R^3$, where $\epsilon>0$ is any positive number. We note that a compactly supported function $v_1$ on $\R^3$, which is considered in the present paper, does not satisfy the above condition (\ref{slow-decay-2}).

Kubo \cite{K04} studied the Cauchy problem (\ref{eq_wave2}) with the critical exponent $\zeta=\zeta_c=2$ in three spatial dimension $n=3$ and proved a small data global existence result for the data $(v_0,v_1)$ satisfying the spatial decay condition (\ref{slow-decay-1}) with the scaling subcritical exponent $\nu\in \left(\frac{3}{2},2\right)$. We note that if the support of data $(v_0,v_1)$ is compact, it is obvious that the function $(v_0,v_1)$ satisfies this decay condition. From his result \cite{K04}, we see that there exists a unique global solution to (\ref{eq_wave2}) in the critical case $\zeta=\zeta_c=2$ and three dimensional case $n=3$ for small initial data decaying rapidly as $|x|\rightarrow\infty$, whereas from the Strauss conjecture about the power nonlinearity $|v|^p$, we see that a local solution to (\ref{eq_wave1}) with the critical exponent $p=p_0(n)$ can not be extended globally for some positive data $(v_0,v_1)$ even if $(v_0,v_1)$ is small with respect to a certain norm, and has a compact support.

We remark that Karageorgis and Tsutaya \cite{KT} reported a small data blow-up result of the problem (\ref{eq_wave2}) in the critical case $\zeta=\zeta_c=2$ and three dimensional case $n=3$ for suitable data $(v_0,v_1)$ satisfying the slowly decaying condition (\ref{slow-decay-2}) as $|x|\rightarrow\infty$ with the scaling critical decay exponent $\nu=\nu_c=\frac{3}{2}$.

Next we recall several results for the Cauchy problem to the wave equation with the scale invariant damping, i.e. $\frac{\mu}{1+t}\partial_t v$ and the power type nonlinearity $|v|^p$:
\EQS{\label{IVP-scale}
  \begin{cases}
    \dt^2 v-\Delta v+\displaystyle{\frac{\mu}{1+t}}\dt v
      =|v|^p, & (x,t)\in\R^n\times [0,T),\\
      v(x,0)=v_0(x), &x\in\R^n,\\
      \dt v(x,0)=v_1(x), &x\in\R^n.
  \end{cases}
}
Existence, blow-up and asymptotic behavior in time of solutions for the problem (\ref{IVP-scale}) have been extensively studied (see \cite{D15, Wakasugi14, LTW17, IkeSoba, ISWk19, TL1, TL2} for example). We only recall results closely related to the present study ($\mu=2$). In studying the problem (\ref{IVP-scale}) with a specific constant $\mu=2$, the Liouville transform $v\mapsto u$ given by
\begin{equation}
\label{Liou}
   u(x,t):=(1+t)^{\frac{\mu}{2}}v(x,t)
\end{equation}
is useful, where $v=v(x,t)$ is a solution to the problem \eqref{IVP-scale} on $\R^n\times [0,T)$. Then the transformed function $u$ satisfies the following equations:
\EQS{\label{IVP-1}
  \begin{cases}
    \dt^2 u-\Delta u+\displaystyle{\frac{\mu(2-\mu)}{4(1+t)^2}} u =\frac{|u|^p}{(1+t)^{\mu(p-1)/2}}, &(x,t)\in\R^n\times [0,T),\\
    u(x,0)=v_0(x), &x\in\R^n,\\
    \displaystyle \dt u(x,0)=\frac{\mu}{2}v_0(x)+v_1(x), &x\in\R^n.
  \end{cases}
}
When $\mu=0$ or $\mu=2$, the mass term $\frac{\mu(2-\mu)}{4(1+t)^2}u$ vanishes. Especially when $\mu=2$, the first equation of (\ref{IVP-1}) becomes the classical wave equation with a power nonlinearity $|u|^p$ with an additional time decay facter $\frac{1}{(1+t)^{p-1}}$, that is, the transformed function $u$ satisfies
\[
     \partial_t^2u-\Delta u=\frac{1}{(1+t)^{p-1}}|u|^p
\]
for $(x,t)\in \R^n\times[0,T)$. In the case of $\mu=2$, by using this transformation it is proved in \cite{Wakasugi14, D15, DLR15} that the critical exponent $p_c=p_c(n)$ for $n=1,2,3$, which divides global existence and blow-up for small solutions to the problem (\ref{IVP-1}) for smooth data $(v_0,v_1)$ decaying rapidly as $|x|\rightarrow\infty$, is given by
\[
p_c(n)=\max\{ p_F(n),p_0(n+2)\}.
\]
Here $p_0$ is the Strauss exponent given by (\ref{strauss}) and $p_F=p_F(n)$ for $n\in \N$ is defined by $p_{F}(n):=1+\frac{2}{n}$
and is called the Fujita exponent, which is the $L^1$-scaling critical exponent for the semilinear heat equation
\[
\dt \theta-\Delta \theta=\theta^p,\ \ \ (x,t)\in \R^n\times [0,T),
\]
where $\theta=\theta(x,t)\ge 0$ is a positive solution on $\R^n\times [0,T)$. See \cite{LTW17, IkeSoba, ISWk19, TL1, TL2} for related works with general positive $\mu$.

We turn back to the original problem (\ref{IVP}).
There is the only one previous study \cite{ITW} about existence of global solutions to the problem (\ref{IVP}) as the authors know. In Theorem 2.1 in \cite{ITW}, a small data blow-up result and the sharp upper estimate of lifespan to the problem (\ref{IVP}) with a specific constant $\mu=2$ and data $(v_0,v_1)$ satisfying the slowly decaying condition (\ref{slow-decay-2}) with the scaling supercritical exponent $\nu<\frac{n-\gamma}{2}$ in multi-dimensional case $n\in \N$. However compactly supported functions $v_0$ and $v_1$ on $\R^n$, which are considered in the present paper, do not satisfy the slowly decaying condition (\ref{slow-decay-2}).

At the end of the introduction, we mention remarkable points of our results to the problem (\ref{IVP}) with $\mu=2$ in the present paper. We prove large data local existence and uniqueness (Theorem \ref{lwp}) in the wider range $\gamma\in \left(-\frac{1}{2},3\right)$, lower estimate of lifespan (Theorem \ref{lower}) in the wider range $\gamma\in \left(-\frac{1}{2}, 0\right)$ and unique global existence (Theorem \ref{T.1-1}) in the wider range $\gamma\in (0,3]$ than without the dissipative term $\frac{2}{1+t}\partial_tv$. We also prove a small data blow-up result and the almost sharp upper estimate of lifespan (Theorem \ref{T.1}) in the subcritical case $\gamma\in (-\frac{1}{2},0)$ for compactly supported data, which do not follow directly from the previous results (Theorem 2.1 in \cite{ITW} and Theorem 6.1 in \cite{H20}).

\section{Main Results}
\ \ In this section, we state our main results in the present paper. In the following, we always assume that
\begin{equation}
\label{coeffi}
     \mu=2.
\end{equation}
Then in the same manner as the study of the power nonlinearity (\ref{IVP-scale}), by using the Liouville transform $v\mapsto u$ given by (\ref{Liou}) with $\mu=2$, where $v=v(x,t)$ is a solution to the Cauchy problem (\ref{IVP}) on $\R^3\times [0,T)$, the transformed function $u$ satisfies the following equations:
\EQS{\label{IVP-2}
  \begin{cases}
    \dt^2 u-\Delta u=\displaystyle{\frac{(V_\ga*u^2)u}{(1+t)^2}}, &(x,t)\in\R^3\times [0,T),\\
    u(x,0)=v_0(x), &x\in\R^3,\\
    \dt u(x,0)=v_0(x)+v_1(x), &x\in\R^3.
  \end{cases}
}
Since the original problem (\ref{IVP}) with $\mu=2$ is equivalent to that of (\ref{IVP-2}), we consider the transformed problem (\ref{IVP-2}) in the following.

The integral equation on $\R^3\times [0,T)$ associated with the Cauchy problem (\ref{IVP-2}) is
\begin{equation}
\label{IE_u_depend}
u(x,t)=u^0(x,t)+L\left((V_{\gamma}*u^2)u\right)(x,t),
\end{equation}
where the function $u^0:\R^3\times \R\rightarrow \R$ associated with the initial condition of (\ref{IVP-2}) is defined by
\begin{equation}
\label{u^0}
u^0(x,t):=\p_t W\left(v_0|x,t\right)+W\left(v_0+v_1|x,t\right),
\end{equation}
and the integral operator $L$ on $C(\R^3\times [0,T))$ is defined by
\begin{equation}
\label{L}
L(F)(x,t):=
\int_0^tW\left(\frac{F(\cdot,s)}{(1+s)^2}
\middle|x,t-s\right)ds,
\end{equation}
where $F\in C(\R^3\times[0,T))$. Here $W$ is the solution operator to the free wave equation, which is defined by
\begin{equation}
\label{solop}
W(\phi|x,t):=
\frac{|t|}{4\pi}\int_{|\omega|=1}\phi(x+|t|\omega)dS_{\omega},
\end{equation}
for $\phi\in C(\R^3)$, where $(x,t)\in \R^3\times \R$ and $dS_{\omega}$ denotes the area element of the two dimensional unit sphere $S_2:=\left\{\omega\in \R^3\ :\ |\omega|=1\right\}$ in $\R^3$.

\begin{rem}
From the representation (\ref{u^0}) and (\ref{solop}), we see that if the initial data $(v_0,v_1)\in C_0^2(\R^3)\times C_0^1(\R^3)$ has a compact support, that is,
\[
      \supp(v_0,v_1)\subset \left\{x\in \R^3\ :\ |x|\le R\right\},
\]
for some $R>0$, then the support of $u^0$ satisfies
\begin{equation}
\label{supportc}
{\rm supp}\ u^0\subset A(T,R):=\{(x,t)\in\R^3\times[0,\infty)\ :\ t-R\le |x|\le t+R\}.
\end{equation}
\end{rem}

Next we give the definition of solution to the Cauchy problem (\ref{IVP-2}) and its lifespan $T(\varepsilon)$ for data $(\varepsilon v_0,\varepsilon v_1)$ with a small parameter $\varepsilon>0$:
\begin{definition}[Solution, Lifespan]
\label{def1}
\begin{itemize}
\item (Solution): Let $T>0$ and $(v_0,v_1)\in C^{2}(\R^3)\times C^{1}(\R^3)$. We say that the function $u:\R^n\times [0,T)\rightarrow \R$ is a solution to the Cauchy problem (\ref{IVP-2}) if $u$ belongs to the class $C(\R^3\times [0,T))$ and it satisfies the integral equation (\ref{IE_u_depend}).
\item (Lifespan): We call the maximal existence time $T=T(v_0,v_1)$ to be lifespan. For initial data $(\varepsilon v_0,\varepsilon v_1)$ with $\varepsilon>0$, the lifespan $T=T(\varepsilon v_0,\varepsilon v_1)$ is denoted by $T(\varepsilon)$, namely
\begin{align*}
    T(\varepsilon)
    &:=\sup\left\{T\in (0,\infty]:\right.\\
    &\left.\text{there exists a unique solution $u$ to (\ref{IVP-2}) with $(\varepsilon v_0,\varepsilon v_1)$ on $\R^n\times [0,T)$}\right\}.
\end{align*}
\end{itemize}
\end{definition}

\subsection{Existence and Uniqueness results}

\ \ In this subsection, we state our existence and uniqueness results to (\ref{IVP-2}). The following proposition means large data local existence to the problem (\ref{IE_u_depend}) with $\gamma\in \left(-\frac{1}{2},3\right)$ and data $(v_0,v_1)\in C^2_0(\R^3)\times C^1_0(\R^3)$ with a compact support. For a Banach space $(\mathcal{X},\|\cdot\|_{\mathcal{X}})$ and $r>0$, we denote by $B_r(\mathcal{X})$ a closed ball in $\mathcal{X}$ with a radius $r$ centered at the origin, that is
\[
     B_r(\mathcal{X}):=\left\{(v_0,v_1)\in \mathcal{X}\ :\ \|(v_0,v_1)\|_{\mathcal{X}}\le r\right\}.
\]
\begin{thm}[Local existence, Uniquenss]
\label{lwp}
Let $\gamma\in \left(-\frac{1}{2},3\right)$, $R\ge 1$ and $(v_0,v_1)\in C_0^2(\R^3)\times C_0^1(\R^3)$ with
\begin{equation}
\label{supp}
{\rm supp}\ (v_0,v_1) \subset\left\{x\in\R^3\ :\ |x|\le R\right\}.
\end{equation}
Then
the following statements hold:
\begin{itemize}
\item (Existence): For any $r>0$, there exists a positive constant $T=T(r)>0$ such that for any initial data $(v_0,v_1)\in B_r(C^2(\R^3)\times C^1(\R^3))$ such that there exists a solution $u\in X(T,M)\subset C(\R^3\times[0,T))$ to the problem (\ref{IVP-2}) such that the solution $u$ has the finite propagation speed:
\begin{equation}
\label{finitp}
   {\rm supp}\ u\subset P(T,R):=\left\{(x,t)\in \R^3\times[0,T)\ :\ |x|\le t+R\right\}.
\end{equation}
Here $M=M(r)$ and $X(T,M)$ are defined by (\ref{closedb}).
\item (Uniqueness): Let $u\in C(\R^3\times [0,T))$ be the solution to (\ref{IVP-2}) obtained in the Existence part. Let $T_1\in (0,T(r)]$ and $w\in C(\R^3\times [0,T_1))$ be another solution to (\ref{IVP-2}). If the identity
\[
(w(x,0),\partial_tw(x,0))=(u(x,0),\partial_tu(x,0))
\]
holds for any $x\in \R^3$, then the identity $w(t,x)=u(t,x)$ holds for any $(x,t)\in \R^3\times [0,T_1)$.
\item (Continuity of the flow map): The flow map
\[
\Omega:B_r(C^2(\R^3)\times C^1(\R^3))\mapsto X(T,M),\ (u(x,0),\partial_tu(x,0))\mapsto u(t,x)
\]
is Lipshitz continuous.

\item (Positivity): If for any $x\in \R^3$, the estimates $v_0(x)=0$ and $v_1(x)\ge 0$ hold, then the inequality $u(t,x)\ge 0$ holds for any $(x,t)\in \R^3\times [0,T)$.
\end{itemize}
\end{thm}

\begin{rem}
\begin{enumerate}
\item The Existence and Uniqueness parts above imply that the lifespan is positive, that is, $T(v_0,v_1)>0$.
\item We can prove that the local solution $u$ to (\ref{IE_u_depend}) obtained above belongs to $C^2(\R^3\times[0,T))$ and becomes a classical solution to (\ref{IVP-2}).
\end{enumerate}
\end{rem}

The next theorem means a lower estimate of lifespan for small solutions to the problem (\ref{IVP-2}) in the critical or subcritical case $\gamma\in \left(-\frac{1}{2},0\right]$:
\begin{thm}[Lower estimate of the lifespan]
\label{lower}
Besides the assumptions in Theorem \ref{lwp}, we assume $\gamma\in \left(-\frac{1}{2},0\right]$. Then there exist positive constants $\e_0=\e_0(\gamma,R,r)>0$ and $A=A(\gamma,R,r)>0$ such that for any $\varepsilon \in [0,\varepsilon_0]$, the lifespan $T(\varepsilon)$ given in Definition \ref{def1} satisfies the following estimate:
\begin{equation}
T(\varepsilon)\ge \left\{\label{lifespan}
\begin{array}{llll}
A\e ^{\frac{2}{\gamma}}, \ &\mbox{if}& \ \gamma\in \left(-\frac{1}{2},0\right),\\
\exp\left(A\e^{-2}\right), &\mbox{if}& \ \gamma=0.
\end{array}
\right.
\end{equation}
\end{thm}
The optimality of the lower estimate is discussed in Remark \ref{rem2.1} and Corollary \ref{opti} below.

The following theorem means small data global existence to the problem (\ref{IE_u_depend}) and asymptotic behavior of solution as $t\rightarrow\infty$ in the supercritical case $\gamma\in (-0,3)$:

\begin{thm}[Global well-posedness, Dissipation]
\label{T.1-1}
Besides the assumptions in Theorem \ref{lwp}, we assume $\gamma\in (0,3)$. Then there exists a positive constant $\epsilon=\epsilon(\gamma,R)>0$ such that for any initial data $(v_0,v_1)\in B_{\epsilon}(C^2(\R^3)\times C^1(\R^3))$, the lifespan is infinity, that is, $T(v_0,v_1)=\infty$. Moreover, the transformed gobal solution $u\in C(\R^3\times[0,\infty))$ scatters as $t\rightarrow\infty$, that is, there exists a solution $u^+=u^+(x,t)\in C(\R^3\times [0,\infty))$ to the free wave equation $\partial_t^2u^+-\Delta u^+=0$ such that the identity holds:
\begin{align}\label{2-11-2}
    \lim_{t\rightarrow\infty}\sup_{x\in \R^3}(1+t+|x|)\left|u(x,t)-u^+(x,t)\right|=0.
\end{align}
Furthermore, the original solution $v=\frac{1}{1+t}u$ to the problem (\ref{IVP}) with $\mu=2$ has a dissipative structure, that is, the estimate
\begin{equation}
\label{2-12-2}
     \sup_{x\in \R^3}(1+t+|x|)|v(t,x)|\le C_*\epsilon(1+t)^{-1}.
\end{equation}
holds for any $t\ge 0$, where $C_*=C_*(\gamma,R,r)$ is some positive constant.
\end{thm}
\begin{rem}
\label{roledis}
From Theorems \ref{lwp}, \ref{lower} and \ref{T.1-1}, we see that large data local existence $(\gamma\in \left(-\frac{1}{2},3\right))$, lower estimate of the lifespan $(\gamma\in \left(-\frac{1}{2},0\right])$ and small data global existence $(\gamma\in (0,3))$ are valid in the wider range $\gamma$ than the case without the damping $\frac{2}{1+t}\partial_t v$.
\end{rem}

\begin{rem}
\label{rem2.1}
From \cite[Theorem 2.1]{K04}, it can be expected that a small data global existence result to the problem (\ref{IVP-2}) holds in the critical case $\gamma=0$.
\end{rem}

\subsection{Blow-up result}

\ \ In this subsection, we state a blow-up result to the problem (\ref{IVP-2}). The following theorem means a small data blow-up result and the almost sharp upper estimate of the lifespan in the subcritical case $\ga<0$:

\begin{thm}[Small data blow-up and almost sharp upper estimate of the lifespan]
\label{T.1}
Besides the assumptions in Theorem \ref{lwp}, we assume $\gamma\in \left(-\frac{1}{2},0\right)$. Moreover we assume that $v_0\equiv 0$ and $v_1=v_1(x)$ is radially symmetric, $v_1\ge 0$ and $v_1\not\equiv 0$. Then the lifespan $T(v_0,v_1)$ is finite, that is, $T(v_0,v_1)<\infty$. Furthermore, 
for any $\delta>0$, there exist positive constants $\varepsilon_1=\varepsilon_1(\delta,\gamma,R,r)>0$ such that for any $\varepsilon\in (0,\varepsilon_1]$, the lifespan $T(\varepsilon)$ given in Definition \ref{def1} satisfies
\[
    T(\varepsilon)\le B \varepsilon ^{\frac{2}{\gamma}-\delta},
\]
where $B=B(\delta,\gamma,R,r)>0$ is a positive constant independent of $\varepsilon$.
\end{thm}

\begin{rem} From Theorem \ref{T.1-1} and Theorem \ref{T.1}, we see that the critical exponent $\gamma_c$ to the problem (\ref{IVP-2}) for compactly supported small data which divides global existence and blow-up is $0$, that is
\[
\gamma_c=0.
\]
\end{rem}

\begin{cor}[Almost optimality of the estimates of lifespan]
\label{opti}
We assume the same assumptions as Theorem \ref{T.1}. Then there exists $\varepsilon_3=\varepsilon_3(\delta,\gamma,R,r)>0$ such that for any $\varepsilon\in (0,\varepsilon_3]$, the lifespan $T(\varepsilon)$ given in Definition \ref{def1} satisfies
\[
     A\varepsilon^{\frac{2}{\gamma}}\le T(\varepsilon)\le B\varepsilon^{\frac{2}{\gamma}-\delta},
\]
where $\delta>0$ is an arbitrary positive number, $A$ and $B$ are positive constants given in Theorem \ref{lower} and Theorem \ref{T.1} respectively.
\end{cor}

The rest of this paper is organized as follows.
In section \ref{S3}, we give a proof of the existence and uniqueness results (Theorems \ref{lwp}, \ref{lower}, \ref{T.1-1}). In subsection \ref{lemmas}, we collect several fundamental lemmas. In subsection \ref{solsp}, we introduce a solution space and a nonlinear mapping associated with the integral equation (\ref{IE_u_depend}) on the space. In subsection \ref{nonli}, we estimate the convolution term (Theorem \ref{lem:potential}) and the Duhamel term (Proposition \ref{lm:apriori1}). In subsection \ref{existe}, we completes the proof of Theorems \ref{lwp}, \ref{lower}, \ref{T.1-1}. In section 4, we give a proof of the blow-up result (Theorem \ref{T.1}).

\section{Proof of the existence results}
\label{S3}

In this section, we give a proof of the existence and uniqueness results (Theorems \ref{lwp}, \ref{lower}, \ref{T.1-1}).

\subsection{Useful lemmas}
\label{lemmas}
\ \ In this subsection, we prepare several lemmas to prove existence and uniqueness of solution to the problem (\ref{IVP-2}).

The following lemma means a  fundamental identity for spherical means:
    \begin{lem}
    \label{lm:Planewave}
    Let $b:(0,\infty)\rightarrow\R$ be a continuous function. Then for any $\rho>0$ and $x\in \R^3$ with $r=|x|$, the identity holds:
    \begin{equation}
    \label{Planewave}
    \begin{array}{ll}
    \d \int_{|\omega|=1}b(|x+\rho \omega|)dS_\omega
    \d = \frac{2\pi}{r\rho }\int_{|\rho-r|}^{\rho+r}\lambda b(\lambda)
    d\lambda.
    \end{array}
    \end{equation}
    \end{lem}
For the proof of this lemma, see Chapter I in \cite{J55} (see also \cite[Lemma 2.1]{K04}).

%

The following lemma means the sharp space-time decay estimate in a point-wise sense for the solution to the free wave equation with compactly supported initial data:
\begin{lem}[Space-time decay estimate for free solution]
\label{lem:decay_est}
Let $(v_0,v_1)\in C^2_0(\R^3)\times C^1_0(\R^3)$ satisfy the condition on the support (\ref{supp}) and $u^0:\R^3\times \R\rightarrow \R$ be the solution to the free wave equation, which is expressed by (\ref{u^0}).
Then there exists a positive constant
$C_{0}=C_0(R)>0$ depending only on $R$ such that for any $(x,t)\in \R^3\times [0,\infty)$ with $t-R\le |x|\le t+R$, the estimate
\begin{equation}
\label{linear est}
(t+|x|+R)|u^0(x,t)|\\
\le\d C_0\sup_{x\in \R^3}\bigg\{\sum_{|\alpha|\le 2}\left|\partial_x^{\alpha}v_0(x)\right|+\sum_{|\beta|\le 1}\left|\partial_x^{\beta}v_1(x)\right|\bigg\}
\end{equation}
holds, where $\alpha\in (\N\cup\{0\})^3$ and $\beta\in (\N\cup\{0\})^3$ denote multi-indices.
\end{lem}
For the proof of this lemma, see \cite[Lemma 2.4]{AKT00} (see also \cite{J79}).

The following lemma means elementary estimates for the solution operator $W$, which is given by (\ref{solop}):
\begin{lem}[Estimates of the solution operator $W$]
\label{lem4.2}
\begin{enumerate}
\item
  Let $\Phi \in C(\R^3)$ and $\phi \in C([0,\I))$.
  If the inequality $|\Phi(x)|\le \phi(|x|)$ holds for any $x\in\R^3$, then the estimate
  \EQS{
  |W(\Phi|x,t)|\le \frac{1}{2r}\int_{|r-t|}^{r+t}\la \phi(\la) d\la
  }
holds for any $(x,t)\in\R^3\times [0,\I)$ with $r=|x|$.
\item
  Let $T>0$, $\Psi\in C(\R^3\times [0,T))$ and $\psi\in C([0,\I)\times [0,T))$. We assume that the estimate $|\Psi(x,t)|\le \psi(|x|,t)$ holds for any $(x,t)\in \R^3\times [0,T)$.
  Then the estimate
  \begin{equation}
  \label{3-4-1}
    \left|\int_0^t W\left(\Psi(\cdot,s)|x,t-s\right)ds\right|
  \le \frac{1}{2r}\iint_{D(r,t)}\la \psi(\la,s)d\la ds,
  \end{equation}
  holds for any $(x,t)\in\R^3\times [0,T)$ with $r=|x|$, where $D(r,t)$ is defined by
  \EQS{
  D(r,t):=\left\{(\la,s)\in[0,\I)^2\ :\ s\in [0,t], |r-(t-s)|\le\la\le r+t-s\right\}.
  }
\end{enumerate}
\end{lem}

For the proof of this lemma, see \cite[Lemma 2.2]{K04}.

The next corollary means a useful representation of the right hand side of the equation (\ref{3-4-1}) through changing variables with $\alpha=s+\lambda$ and $\beta=s-\lambda$, which is equivalent to $s=\frac{\alpha+\beta}{2}$ and $\lambda=\frac{\alpha-\beta}{2}$.

\begin{cor}[A representation of the integral on $D(r,t)$]
\label{repre}
We assume the same assumptions of 2 in Lemma \ref{lem4.2}. Let $(x,t)\in \R^3\times [0,T)$ with $|x|=r$. Then the estimates
  \begin{align}
    \left|\int_0^t W\left(\Psi(\cdot,s)|x,t-s\right)ds\right|
  &\le \frac{1}{2r}\iint_{D(r,t)}\la \psi(\la,s)d\la ds\notag\\
  &=\frac{1}{4r}\int_{\mathcal{D}(r,t)}\frac{\alpha-\beta}{2}\psi\left(\frac{\alpha-\beta}{2},\frac{\alpha+\beta}{2}\right)d\alpha d\beta
\label{3-7-2}
  \end{align}
hold, where $\mathcal{D}(r,t)$ is defined by (\ref{3-8-8}). Moreover, under the condition $r\le t+R$, the estimate
\begin{equation}
\label{3-8-2}
  \left|\int_0^t W\left(\Psi(\cdot,s)|x,t-s\right)ds\right|\le \frac{1}{4r}\int_{-R}^{t-r}\int_{|t-r|}^{t+r}\frac{\alpha-\beta}{2}\psi\left(\frac{\alpha-\beta}{2},\frac{\alpha+\beta}{2}\right)d\alpha d\beta
\end{equation}
holds.
\end{cor}
This corollary is proved as follows: The identity $d\lambda ds=\mathcal{J}d\alpha d\beta$ holds, where $\mathcal{J}$ is defined by
\[
   \mathcal{J}:=\left|\det
   \left(
\begin{array}{cc}
\frac{ds}{d\alpha} & \frac{ds}{d\beta} \\
\frac{d\lambda}{d\alpha} & \frac{d\lambda}{d\beta}
\end{array}
\right)
\right|
=\left|\det
   \left(
\begin{array}{cc}
\frac{1}{2} & \frac{1}{2} \\
\frac{1}{2} & -\frac{1}{2}
\end{array}
\right)
\right|=\frac{1}{2}.
\]
We remark that for any $(x,t)\in \R^3\times [0,T)$ with $|x|=r$, $(\lambda,s)\in D(r,t)$ is equivalent to $(\alpha,\beta)\in \mathcal{D}(r,t)$, where $\mathcal{D}(r,t)$ is defined by
\begin{equation}
\label{3-8-8}
  \mathcal{D}(r,t):=\begin{cases}
    \mathcal{D}_1(r,t)\cup \mathcal{D}_2(r,t), &\text{if}\ t\ge r,\\
    \left\{(\alpha,\beta)\ :\ -t-r\le \beta\le t-r,\ -\beta\le \alpha\le r+t \right\}, &\text{if}\ t<r,
  \end{cases}
\end{equation}
with
\begin{align*}
  \mathcal{D}_1(r,t):&=\left\{(\alpha,\beta)\ :\ r-t\le \beta\le t-r,\ t-r\le \alpha\le r+t\right\},\notag\\
  \mathcal{D}_2(r,t):&=\left\{(\alpha,\beta)\ :\ -r-t\le \beta\le t-r,\ -\beta\le \alpha\le r+t\right\}.\notag
\end{align*}
We also note that the inclusion $\mathcal{D}(r,t)\subset[|t-r|,t+r]\times[-R,t-r]$ holds under the condition $r\le t+R$.

We use a notation $\LR{t}:=1+|t|$ for $t\in\R$.

\begin{lem}\label{lem4.3}
  Let $\ka\in\R$. Then for any $(r,t)\in [0,\infty)^2$, the estimate
  \EQQS{
  \int_{|r-t|}^{r+t}(1+\lambda)^{-(\ka+1)} d\lambda
  \le
  \begin{cases}
    \d \frac{2\max(1,\kappa)}{\kappa}\cdot\frac{\min(r,t)}{\LR{t+r}\LR{t-r}^\ka}, &\quad\mbox{if}\quad \kappa>0,\\
    \d \log \frac{\langle t+r\rangle}{\langle t-r\rangle}, &\quad\mbox{if}\quad \kappa=0,\\
    \d \frac{2\max(1,-\kappa)}{-\kappa}\cdot\frac{\min(r,t)}{\langle t+r\rangle^{\ka+1}}, &\quad\mbox{if}\quad \kappa<0\\
  \end{cases}
  }
holds.
\end{lem}

This lemma can be proved by a direct computation and the following elementary inequality: For any $\vartheta>0$, the estimate
\begin{equation}
\label{eleme}
1-a^{\vartheta}\le \max(\vartheta,1)(1-a)
\end{equation}
holds for any $a\in [0,1]$.

The next lemma is employed to prove lower estimate of the lifespan in the special case of $\gamma=2$.
\begin{lem}\label{lem4.4}
  Let $\ka>0$. Then there exists a positive constant $C=C(\kappa)>0$ such that for any $(r,t)\in [0,\infty)^2$, the estimate
\begin{equation*}
    \int_{|r-t|}^{r+t}(1+\lambda)^{-\kappa+1}\log (2+\lambda)d\lambda\le \frac{C\min(r,t)\left\{\log(1+\langle t-r\rangle)\right\}}{\langle t+r\rangle\langle t-r\rangle^{\kappa}}
\end{equation*}
holds.
\end{lem}

For the proof of this lemma, see Lemma 2.3 in \cite{K04}.

\subsection{Introduction of the solution space $X(T)$}
\label{solsp}
\ \ For $T\in (0,\infty]$ and $R>0$, we introduce the solution space $X(T)=X(T,R)$ to the problem (\ref{IVP-2}) with the data $(v_0,v_1)\in C^2(\R^3)\times C^1(\R^3)$ satisfying the condition (\ref{supp}) given by
\begin{equation}
\label{sols}
 X(T)=X(T,R):=\{u\in C\left(\R^3\times[0,T)\right):\|u\|_{X(T)}<\I,\ \supp u\subset P(T,R)\},
\end{equation}
where the set $P(T,R)$ is defined by (\ref{finitp}) and the norm $\|\cdot\|_{X(T)}:X(T)\rightarrow \R_{\ge 0}$ is defined by
\[
	\|u\|_{X(T)}
	 :=\sup_{(x,t)\in\R^3\times[0,T)}
	  \tau_{+}(|x|,t)N_{\gamma}(\tau_{-}(|x|,t))|u(x,t)|.
\]
Here the weight functions $\tau_{\pm}:\R_{\ge 0}\times \R_{\ge 0}\rightarrow \R$ are given by
\[
\tau_{\pm}(r,t):=\frac{t\pm r+2R}{R},
\]
where the double-sign corresponds. Moreover for $\gamma\in \left(-\frac{1}{2},3\right)$, the function $N_{\gamma}:\R_{\ge0}\rightarrow \R_{\ge 0}$ is defined by
\begin{equation}
\label{wightf}
    N(\varrho)=N_{\gamma}(\varrho):=
    \left\{
\begin{array}{lll}
{\varrho}^{\gamma+1},\  &\mbox{if}&  \ \gamma\in \left(-\frac{1}{2},2\right),\\
\displaystyle\frac{{\varrho}^{3}}{\log (1+\varrho)}, \ &\mbox{if}& \ \gamma=2,\\
{\varrho}^3,\  &\mbox{if}&  \ \gamma\in \left(2,3\right).
\end{array}
\right.
\end{equation}

\begin{rem}
For $T\in (0,\infty]$ and $R>0$, the space $\left(X(T,R),\|\cdot\|_{X(T)}\right)$ is a Banach space.
\end{rem}

In order to construct a solution to the integral equation (\ref{IE_u_depend}) on the Banach space $X(T)$, we introduce the nonlinear mapping $\Gamma$ defined by
\begin{equation}
\label{nonlima}
     \Gamma[u](x,t):=u^0(x,t)+L\left((V_{\gamma}*u^2)u\right)(x,t),
\end{equation}
for $u\in X(T)$, where the right-hand side of (\ref{nonlima}) is same as that of (\ref{IE_u_depend}). In subsection \ref{existe} below, we discuss whether the nonlinear mapping $\Gamma$ is a contraction mapping from a closed ball $X(T,M)$ (for the definition, see \ref{closedb}) in $X(T)$ into itself.

\subsection{Multilinear estimates}
\label{nonli}

\ \ In this subsection, we estimate the convolution term in a point-wise sense (Theorem \ref{lem:potential}) and the Duhamel term (Proposition \ref{lm:apriori1}).

\ \ The next theorem means boundedness of the convolution term from $L_{x,t}^{\infty}(\R^3\times[0,T))$ to $X(T)\times X(T)$:
\begin{thm}[A bilinear estimate from $L^{\infty}_{x,t}(\R^3\times [0,T))$ to $X(T)\times X(T)$]
\label{lem:potential}
Let $T\in (0,\infty]$ and $\ga\in (-\frac{1}{2},3)$. Then there exists a positive constant $C_1=C_1(\gamma)>0$ depending only on $\gamma$ such that for any $u,w\in X(T)$, the estimate
\begin{align}
\label{potential}
\left|\left(V_{\gamma}*(uw)\right)(x,t)\right|
\le C_1\|u\|_{X(T)}\|w\|_{X(T)}\{\mathcal{W}_R(|x|,t)\}^{-1}
\end{align}
holds for any $(x,t)\in P(T,R)$, where $\mathcal{W}_R:\R_{\ge 0}\times \R_{\ge 0}\rightarrow \R_{\ge 0}$ is a weight function defined by
\[
\mathcal{W}_R(r,t):=\begin{cases}
  R^{\gamma-3}\left\{\tau_{+}(r,t)\right\}^{2}, &\ \mbox{if}
\ \gamma\in (2,3), \\
  R^{\gamma-3}\left\{\tau_{+}(r,t)\right\}^{\gamma}, &\ \mbox{if}
\ \gamma\in \left(-\frac{1}{2},2\right), \\
  R^{-1}\log(1+R)\left\{\tau_{+}(r,t)\right\}^{2}\left\{\log\left(1+\tau_+(r,t)\right)\right\}^{-1}, &\
\mbox{if}\ \gamma=2.
\end{cases}
\]
\end{thm}

\begin{proof}[Proof of Theorem \ref{lem:potential}]
Let $(x,t)\in P(T,R)$. Set $r:=|x|$. We divide the proof into the three cases where $\gamma\in \left(-\frac{1}{2},2\right)$, $\gamma\in (2,3)$ and $\gamma=2$.\\
{\bf Case1.}\ $\gamma\in \left(-\frac{1}{2},2\right)$: In this case, the identity $N_{\gamma}(\varrho)=\rho^{\gamma+1}$ holds. We note that by $\gamma<2$, the estimate (\ref{eleme}) with $\vartheta=2-\gamma>0$, the inequalities
\begin{align}
    \int_{|r-\rho|}^{r+\rho}\lambda^{1-\gamma}d\lambda
    \le\frac{2\min(1,2-\gamma)}{2-\gamma}\min(r,\rho)(r+\rho)^{1-\gamma}
    \label{3-12}
\end{align}
hold for $\rho\ge 0$. We also note that if the inequalities $\rho\ge 0$ and $r\le t+R$ hold, the estimate
\begin{equation}
\label{3-13}
    t+\rho+2R\ge \frac{1}{2}(t+r+2R)
\end{equation}
holds, which implies that the inequality $\tau_{+}(\rho,t)\ge \frac{1}{2}\tau_{+}(r,t)$ holds. By the definition of $X(T)$-norm, changing variables with $P(T,R)\ni y=\rho\omega$ ($\rho\in \left(0,t+R\right)$ and $\omega\in S_2$), Lemma \ref{Planewave} with $b(r)=|r|^{-\gamma}$, the estimate (\ref{3-12}) and the above estimate, the estimates
\begin{align}
    &\text{L.H.S of (\ref{potential})}\le \int_{P(T,R)}\frac{\left|u(y,t)\right|\left|w(y,t)\right|}{|x-y|^\ga}dy\notag\\
    &=\|u\|_{X(T)}\|w\|_{X(T)}\frac{2\pi}{r}\int_{0}^{t+R}\frac{\rho}{\left\{\tau_+(\rho,t)\right\}^2\left\{\tau_-(\rho,t)\right\}^{2(\gamma+1)}}\left(\int_{|r-\rho|}^{r+\rho}\lambda^{1-\gamma}d\lambda\right)d\rho\notag\\
  &\le\|u\|_{X(T)}\|w\|_{X(T)}\frac{4\pi\min(1,2-\gamma)}{2-\gamma}\cdot\frac{1}{r}\int_{0}^{t+R}\frac{\rho\min(r,\rho)(r+\rho)^{1-\gamma}}{\left\{\tau_+(\rho,t)\right\}^2\left\{\tau_-(\rho,t)\right\}^{2(\gamma+1)}}d\rho\notag\\
  &=\|u\|_{X(T)}\|w\|_{X(T)}\frac{16\pi\min(1,2-\gamma)}{2-\gamma}R^{2-\gamma}\left\{\tau_+(r,t)\right\}^{-\gamma}\int_{0}^{t+R}\left\{\tau_-(\rho,t)\right\}^{-2(\gamma+1)}d\rho\notag\\
  &\le \|u\|_{X(T)}\|w\|_{X(T)}\frac{16\pi\min(1,2-\gamma)\left(1-3^{-2\gamma-1}\right)}{(2-\gamma)(2\gamma+1)}R^{3-\gamma}\left\{\tau_+(r,t)\right\}^{-\gamma}\notag
\end{align}
hold, which implies the estimate (\ref{potential}).\\
{\bf Case2.}\ $\gamma\in \left(2,3\right)$: In this case, the identity $N_{\gamma}(\varrho)=\varrho^3$ holds. We divide the integral region $P(T,R)$ into the two pieces $P_{\le}(T,R):=P(T,R)\cap \left\{y\in \R^3:|x-y|\le \frac{R}{2}\right\}$ and $P_{>}(T,R):=P(T,R)\cap \left\{y\in \R^3:|x-y|>\frac{R}{2}\right\}$. We note that for any $y\in P_{\le}(T,R)$, the estimates
\[
   t+|y|+2R\ge t+|x|+\frac{3}{2}R\ge \frac{3}{4}(t+|x|+2R)
\]
hold, which implies the estimate $\tau_{+}(|y|,t)\ge \frac{3}{4}\tau_+(|x|,t)$ holds. We also note that for any $y\in A_{\le}(T,R)$, the estimates
\begin{equation}
\label{3-15-1}
   t-|y|+2R\ge t-|x|+\frac{3}{2}R\ge \frac{3}{4}(t-|x|+2R)\ge \frac{3}{4}R
\end{equation}
hold, which implies the estimate $\tau_{-}(|y|,t)\ge \frac{3}{4}\tau_{-}(|x|,t)\ge \frac{3}{4}$ holds. By the definition of $X(T)$-norm, the above estimates of the weight functions $\tau_{\pm}$ and the assumption $\gamma<3$, the estimates
\begin{align}
  &\int_{P_{\le}(T,R)}\frac{\left|u(y,t)\right|\left|w(y,t)\right|}{|x-y|^\ga}dy\notag\\
    &\le\|u\|_{X(T)}\|w\|_{X(T)}\int_{P_{\le}(T,R)}\frac{1}{\left\{\tau_+(|y|,t)\right\}^2\left\{\tau_-(|y|,t)\right\}^{6}|x-y|^{\gamma}}dy\notag\\
    &\le\|u\|_{X(T)}\|w\|_{X(T)}\left(\frac{4}{3}\right)^8\left\{\tau_+(|x|,t)\right\}^{-2}\int_{P_{\le}(T,R)}|x-y|^{-\gamma}dy\notag\\
    &\le \|u\|_{X(T)}\|w\|_{X(T)}\frac{4^8\cdot2\pi}{3^8(3-\gamma)2^{3-\gamma}}R^{3-\gamma}\left\{\tau_+(|x|,t)\right\}^{-2}\label{3-14}
\end{align}
hold. We note that for $y\in P_{>}(T,R)$, the estimate
\[
   |x-y|\ge \frac{1}{4}(|x-y|+R)
\]
holds, which implies that the inequality $|x-y|^{-\gamma}\le 4^{\gamma}(|x-y|+R)^{-\gamma}$ holds due to $\gamma\ge0$. By changing variables with $y=\rho\omega$ ($\rho\in \left(0,t+R\right)$), Lemma \ref{Planewave} with $b(r)=|r+R|^{-\gamma}$, Lemma \ref{lem4.3} with $\kappa=\gamma-2>0$ and the estimate (\ref{3-13}), the inequalities
\begin{align}
      &\int_{P_{>}(T,R)}\frac{\left|u(y,t)\right|\left|w(y,t)\right|}{|x-y|^\ga}dy\notag\\
      &\le\|u\|_{X(T)}\|w\|_{X(T)}4^{\gamma}\int_{P(T,R)}\frac{1}{\left\{\tau_+(|y|,t)\right\}^2\left\{\tau_-(|y|,t)\right\}^{6}(|x-y|+R)^{\gamma}}dy\notag\\
  &\le\|u\|_{X(T)}\|w\|_{X(T)}\frac{(4\pi)4^{\gamma}\min(1,\gamma-2)}{\gamma-2}\notag\\
  &\ \times\frac{1}{r}\int_{0}^{t+R}\frac{\rho\min(r,\rho)}{\left\{\tau_+(\rho,t)\right\}^2\left\{\tau_-(\rho,t)\right\}^{6}(\rho+r+R)(|\rho-r|+R)^{\gamma-2}}d\rho\notag\\
  &\le\|u\|_{X(T)}\|w\|_{X(T)}\frac{(4\pi)4^{\gamma}\min(1,\gamma-2)}{\gamma-2}R^{2-\gamma}\left\{\tau_+(r,t)\right\}^{-2}\int_{0}^{t+R}\left\{\tau_-(\rho,t)\right\}^{-6}d\rho\notag\\
  &\le \|u\|_{X(T)}\|w\|_{X(T)}\frac{(4\pi)4^{\gamma}\min(1,2-\gamma)\left(1-3^{-5}\right)}{5(\gamma-2)}R^{3-\gamma}\left\{\tau_+(r,t)\right\}^{-2}\label{3-15}
\end{align}
hold. The estimates (\ref{3-14}) and (\ref{3-15}) imply the estimate (\ref{potential}).\\
{\bf Case3.}\ $\gamma=2$: In this case, the identity $N_{2}(\varrho)=\frac{\varrho^{3}}{\log(1+\varrho)}$ holds. We note that the function $N_\gamma$ is monotone increasing on $\R_{\ge 0}$. We divide the integral region $P(T,R)$ into the two pieces $P_{\le}(T,R)$ and $P_>(T,R)$ as Case2. For any $y\in P_{\le }(T,R)$, in the same manner as the proof of (\ref{3-15-1}), the estimates
\begin{equation}
    N_{\gamma}(\tau_{-}(|y|,t))\ge N_{\gamma}\left(\frac{3}{4}\tau_{-}(|x|,t)\right)\ge \left(\frac{3}{4}\right)^3N_{\gamma}(\tau_{-}(|x|,t)))
\end{equation}
hold. In the same manner as the proof of (\ref{3-14}), the inequalities
\begin{align}
  &\int_{P_{\le}(T,R)}\frac{\left|u(y,t)\right|\left|w(y,t)\right|}{|x-y|^2}dy\notag\\
    &\le\|u\|_{X(T)}\|w\|_{X(T)}\int_{P_{\le}(T,R)}\frac{1}{\left\{\tau_+(|y|,t)\right\}^2\left\{N(\tau_-(|y|,t))\right\}^{2}|x-y|^{2}}dy\notag\\
    &\le\|u\|_{X(T)}\|w\|_{X(T)}\left(\frac{4}{3}\right)^8\left\{\tau_+(|x|,t)\right\}^{-2}\left\{N(\tau_-(|x|,t))\right\}^{-2}\int_{P_{\le}(T,R)}|x-y|^{-2}dy\notag\\
    &\le \|u\|_{X(T)}\|w\|_{X(T)}\frac{4^8\cdot2\pi}{3^8(3-\gamma)2^{3-\gamma}}R^{3-\gamma}\left\{\tau_+(|x|,t)\right\}^{-2}\left\{N(\tau_-(|x|,t))\right\}^{-2}\label{3-19}
\end{align}
hold. Next we consider the case where $y$ belongs to $P_{>}(T,R)$. We note that the estimates
\begin{equation}
\label{3-21-2}
     \frac{1}{r}\int_{|r-\rho|}^{r+\rho}\lambda\left(\lambda+R\right)^{-2}d\lambda
     \le\frac{1}{r}\log\frac{r+\rho+R}{|r-\rho|+R}\le \frac{6}{\log 2}\cdot\frac{\log(r+\rho+2R)}{r+\rho}
\end{equation}
hold for any $r,\rho>0$. Indeed, the right estimate of (\ref{3-21-2}) is verified as follows. We divide the proof into the two cases, $2r\le \rho$ and $2r>\rho$. We consider the case where $2r\le \rho$. Then the estimate $\rho-r\ge \frac{1}{3}(\rho+r)>0$ holds. By using this estimate, the estimates
\begin{align*}
    \frac{r+\rho+R}{|r-\rho|+R}=1+\frac{2\min(r,\rho)}{R+|r-\rho|}\le 1+\frac{2r}{R+\rho-r}\le 1+\frac{6r}{\rho+r+2R}
\end{align*}
hold. Thus noting that the estimate $\log (1+\theta)\le \theta$ holds for any $\theta\ge 0$, the estimates
\begin{align*}
  \frac{1}{r}\log\frac{r+\rho+R}{|r-\rho|+R}&\le \frac{1}{r}\log\left(1+\frac{6r}{\rho+r+2R}\right)\le \frac{6}{\rho+r+2R}\\
  &\le \frac{6}{\log 2}\cdot\frac{\log(r+\rho+2R)}{r+\rho}
\end{align*}
hold. Next we consider the case where $2r>\rho$. Then the estimate $r>\frac{1}{3}(r+\rho)$ holds. Thus by the estimate $R>1$, the inequality
\[
\frac{1}{r}\log\frac{r+\rho+R}{|r-\rho|+R}\le 3\frac{\log(r+\rho+2R)}{r+\rho}
\]
is valid, which completes the proof of the estimate (\ref{3-21-2}). In the similar manner as the proof of the estimate (\ref{3-15}), by using the estimate (\ref{3-21-2}), the estimates
\begin{align}
           &\int_{P_{>}(T,R)}\frac{\left|u(y,t)\right|\left|w(y,t)\right|}{|x-y|^2}dy\notag\\
    &\le\|u\|_{X(T)}\|w\|_{X(T)}(2\pi)4^{2}\int_{0}^{t+R}\frac{\rho\left\{\log(1+\tau_{-}(\rho,t))\right\}^2}{\left\{\tau_+(\rho,t)\right\}^2\left\{\tau_-(\rho,t)\right\}^{6}}\left\{\frac{1}{r}\int_{|r-\rho|}^{r+\rho}\lambda(\lambda+R)^{-2}d\lambda\right\}d\rho\notag\\
  &\le\|u\|_{X(T)}\|w\|_{X(T)}\frac{(2\pi)4^{2}6}{\log 2}
  \int_{0}^{t+R}\frac{\rho\left\{\log(1+\tau_-(\rho,t))\right\}^2\log(r+\rho+2R)}{\left\{\tau_+(\rho,t)\right\}^2\left\{\tau_-(\rho,t)\right\}^{6}(\rho+r)}d\rho\notag\\
  &\le\|u\|_{X(T)}\|w\|_{X(T)}\frac{(2\pi)4^{2}6}{\log 2}\left(\frac{\log 4R}{\log 2}+1\right)\left\{\tau_+(r,t)\right\}^{-2}\log\left(1+\tau_+(r,t)\right)\notag\\
  &\ \ \ \times\int_{0}^{t+R}\frac{\left\{\log(1+\tau_-(\rho,t))\right\}^2}{\left\{\tau_-(\rho,t)\right\}^{6}}d\rho\notag\\
  &\le\|u\|_{X(T)}\|w\|_{X(T)}\frac{(2\pi)4^{2}6^3}{5^2}\log(1+R)R\left\{\tau_+(r,t)\right\}^{-2}\log\left(1+\tau_+(r,t)\right)\notag
\end{align}
hold, which completes the proof of the theorem.

\end{proof}

The next proposition means boundedness of a trilinear operator associated with the integral operator $L$ given by (\ref{L}) and the convolution term from $X(T)$ to $X(T)\times X(T)\times X(T)$:
\begin{prop}[A trilinear estimate from $X(T)$ to $(X(T))^3$]\label{lm:apriori1}
Let $\gamma\in \left(-\frac{1}{2},3\right)$, $T>0$, $u_1,u_2,u_3\in X(T)$ and $L$ be the integral operator on $C(\R^3\times[0,T))$ given by {\rm (\ref{L})}. Then there exists a positive constant $C_2=C_2(\gamma,R)>0$ depending only on $\gamma,R$ such that the estimate
\EQQ{
&\|L((V_{\gamma}*(u_1u_2))u_3)\|_{X(T)}\\
&\le C_2 D_{\gamma}(T)\prod_{i=1}^3\|u_i\|_{X(T)}
\begin{cases}
  R^{5-\gamma}, &\text{if}\ \gamma\in\left(-\frac{1}{2},3\right)\ \text{and}\ \gamma\ne 2,\\
  R^{3}\log(1+R),&\text{if}\ \gamma=2
\end{cases}
}
holds, where $D_{\gamma}:\R_{>0}\rightarrow \R_{\ge 0}$ is a function defined by
\begin{equation}
\label{D}
D(T)=D_{\gamma}(T):=
\left\{
\begin{array}{llll}
\gamma^{-1}, \ &\mbox{if}& \ \gamma\in (0,3),\\
\d \displaystyle\log\frac{T+2R}{R}, &\mbox{if}& \ \gamma=0,\\
\d (-\gamma)^{-1}\left(\frac{T+R}{R}\right)^{-\gamma}, &\mbox{if}& \
 \gamma\in \left(-\frac{1}{2},0\right).\\
\end{array}
\right.
\end{equation}
\end{prop}

\begin{proof}[Proof of Proposition \ref{lm:apriori1}]
We introduce a function $\Psi:\R^3\times [0,T)\rightarrow\R$ given by
\[
   \Psi(x,t):=\frac{1}{(1+t)^2}\left(V_{\gamma}*(u_1u_2)\right)(x,t)u_3(x,t).
\]
By Theorem \ref{lem:potential} and the definition of $X(T)$-norm, the estimate
\begin{align*}
  |\Psi(x,t)|&\le C_1\prod_{i=1}^3\|u_i\|_{X(T)}(1+t)^{-2}\left\{\mathcal{W}_R(|x|,t)\right\}^{-1}\left\{\tau_+(|x|,t)\right\}^{-1}\left\{N_{\gamma}(\tau_-(|x|,t))\right\}^{-1}\notag\\
  &=:C_1\prod_{i=1}^3\|u_i\|_{X(T)}(1+t)^{-2}\left\{\widetilde{\mathcal{W}}_R(|x|,t)\right\}^{-1}
\end{align*}
holds for any $(x,t)\in \R^3\times[0,T)$, where $C_1$ is a positive constant given in Theorem \ref{lem:potential}. Thus we can apply the estimate (\ref{3-8-2}) in Corollary \ref{repre} to get the estimates
\begin{align}
   &\left|L((V_{\gamma}*(u_1u_2))u_3)(x,t)\right|\notag
   =\left|\int_0^tW(\Psi(\cdot,s)\big|x,t-s)ds\right|\notag\\
   &\le C_1\prod_{i=1}^3\|u_i\|_{X(T)}\notag\\
   &\times{(4r)^{-1}}\int_{-R}^{t-r}\int_{|t-r|}^{t+r}\frac{\alpha-\beta}{2}\left(1+\frac{\alpha+\beta}{2}\right)^{-2}\left\{\widetilde{\mathcal{W}}_R\left(\frac{\alpha-\beta}{2},\frac{\alpha+\beta}{2}\right)\right\}^{-1}d\alpha d\beta\label{3-24}
\end{align}
for $(x,t)\in \R^3\times [0,T)$ with $r=|x|$. We note that for any $s\ge 0$ and $R>1$, the estimate
\[
      1+s\ge \frac{2R+s}{2R}
\]
holds. By this estimate and setting $s:=\frac{\alpha+\beta}{2}\ge 0$ with $\alpha\ge 0$ and $\beta\ge -R$, the estimate
\begin{equation}
\label{3-25}
     1+\frac{\alpha+\beta}{2}\ge \frac{\alpha+2R}{4R}.
\end{equation}
We also remark that the estimates
$\frac{\alpha-\beta}{2}\le \frac{\alpha+R}{2}\le \frac{\alpha+2R}{4}$
hold. By combining this estimate and the estimates (\ref{3-24})-(\ref{3-25}), the inequalities
\begin{align}
   &\left|L((V_{\gamma}*(u_1u_2))u_3)(x,t)\right|\le C_1\prod_{i=1}^3\|u_i\|_{X(T)}\notag\\
   &\times{r^{-1}}R\int_{-R}^{t-r}\int_{|t-r|}^{t+r}\left(\frac{\alpha+2R}{R}\right)^{-1}\left\{\widetilde{\mathcal{W}}_R\left(\frac{\alpha-\beta}{2},\frac{\alpha+\beta}{2}\right)\right\}^{-1}d\alpha d\beta\label{3-26}
\end{align}
hold for $(x,t)\in \R^3\times[0,T)$ with $r=|x|$. We note that by the definitions of the weight functions $\tau_{\pm}$, the identities
\[
    \tau_+\left(\frac{\alpha-\beta}{2},\frac{\alpha+\beta}{2}\right)=\frac{\alpha+2R}{R},\ \ \tau_-\left(\frac{\alpha-\beta}{2},\frac{\alpha+\beta}{2}\right)=\frac{\beta+2R}{R}
\]
hold. We note that for any $\vartheta>1$, by the estimate (\ref{eleme}), the inequalities
\begin{align}
    \int_{t-r}^{t+r}\left(\frac{\alpha+2R}{R}\right)^{-\vartheta}d\alpha
    &=\frac{R}{\vartheta-1}\left\{\tau_-(r,t)\right\}^{-\vartheta+1}\left\{1-\left(\frac{\tau_-(r,t)}{\tau_+(r,t)}\right)^{\vartheta-1}\right\}\notag\\
    &\le \frac{2\max(1,\vartheta-1)}{\vartheta-1}r\left\{\tau_-(r,t)\right\}^{-\vartheta+1}\left\{\tau_+(r,t)\right\}^{-1}\label{3-27}
\end{align}
hold. We also remark that for any $b\in \R$, the estimate
\begin{equation}
\label{3-28}
    \int_{-R}^{t-r}\left(\frac{\beta+2R}{R}\right)^bd\beta
    \le R
    \left\{
\begin{array}{llll}
\displaystyle -\frac{1}{b+1}, \ &\mbox{if}& \ b<-1,\\
\d \displaystyle\log\frac{t+2R}{R}, &\mbox{if}& \ b=-1,\\
\d \frac{1}{b+1}\left(\frac{t+R}{R}\right)^{b+1}, &\mbox{if}& \
 b>-1\\
\end{array}
\right.
\end{equation}
holds. Let $(x,t)\in \R^3\times[0,T)$ with $r=|x|$. We divide the proof into the three cases where $\gamma\in \left(-\frac{1}{2},2\right)$, $\gamma\in (2,3)$ and $\gamma=2$.\\
{\bf Case1.}\ $\gamma\in \left(-\frac{1}{2},2\right)$: In this case, by the estimates (\ref{3-26}), (\ref{3-27}) with $\vartheta=\gamma+2$ and (\ref{3-28}) with $b=-\gamma-1$, the inequalities
\begin{align*}
   &\left|L((V_{\gamma}*(u_1u_2))u_3)(x,t)\right|\\
   &\le C_1\prod_{i=1}^3\|u_i\|_{X(T)}
   {r^{-1}}R^{4-\gamma}\int_{-R}^{t-r}\left(\frac{\beta+2R}{R}\right)^{-\gamma-1}d\beta\int_{t-r}^{t+r}\left(\frac{\alpha+2R}{R}\right)^{-\gamma-2}d\alpha\\
   &\le 2C_1R^{5-\gamma}\prod_{i=1}^3\|u_i\|_{X(T)}\left\{\tau_-(r,t)\right\}^{-\gamma-1}\left\{\tau_+(r,t)\right\}^{-1}D_{\gamma}(T)
\end{align*}
hold.\\
{\bf Case2.}\ $\gamma\in (2,3)$: In this case, by the estimates (\ref{3-26}), (\ref{3-27}) with $\vartheta=4$ and (\ref{3-28}) with $b=-3$, the inequalities
\begin{align*}
   &\left|L((V_{\gamma}*(u_1u_2))u_3)(x,t)\right|\\
   &\le C_1\prod_{i=1}^3\|u_i\|_{X(T)}r^{-1}R^{4-\gamma}\int_{-R}^{t-r}\left(\frac{\beta+2R}{R}\right)^{-3}d\beta\int_{t-r}^{t+r}\left(\frac{\alpha+2R}{R}\right)^{-4}d\alpha\\
   &\le 2C_1R^{5-\gamma}\prod_{i=1}^3\|u_i\|_{X(T)}\left\{\tau_-(r,t)\right\}^{-3}\left\{\tau_+(r,t)\right\}^{-1}
\end{align*}
hold.\\
{\bf Case3.}\ $\gamma=2$: In this case, by the estimates (\ref{3-26}) and Lemma \ref{lem4.4}, the estimates
\begin{align*}
&\left|L((V_{\gamma}*(u_1u_2))u_3)(x,t)\right|\\
   &\le C_1\prod_{i=1}^3\|u_i\|_{X(T)}r^{-1}R^{2}\log(1+R)\int_{-R}^{t-r}\left(\frac{\beta+2R}{R}\right)^{-3}\log\left(1+\frac{\beta+2R}{R}\right)d\beta\\
   &\ \ \ \times \int_{t-r}^{t+r}\left(\frac{\alpha+2R}{R}\right)^{-4}\log\left(1+\frac{\alpha+2R}{R}\right)d\alpha\\
   &\le C_1C(\log 2)R^3\log(1+R)\prod_{i=1}^3\|u_i\|_{X(T)}\left\{\tau_-(r,t)\right\}^{-3}\log\left(1+\tau_-(r,t)\right)\left\{\tau_+(r,t)\right\}^{-1}
\end{align*}
hold, which completes the proof of the proposition.
\end{proof}

\subsection{Proof of Theorems \ref{lwp}, \ref{lower}, \ref{T.1-1}}
\label{existe}

\subsubsection{Proof of Theorem \ref{lwp}}

\ \ In order to prove Theorem \ref{lwp}, we use the following $L^{\infty}(\R^3)$-$L^{\infty}(\R^3)$ estimate for the solution operator $W$ given by (\ref{solop}).

\begin{lem}[$L^{\infty}(\R^3)$-$L^{\infty}(\R^3)$ estimate]
Let $\alpha\in (\N\cup\{0\})^3$ be a multi-index and $\phi:\R^3\rightarrow \R$ be a function satisfying $\partial_x^{\alpha}\phi\in L^{\infty}(\R^3)$. Then the estimate
\begin{equation}
\label{3-22-a}
     \sup_{x\in \R^3}\left|\partial_x^{\alpha}W(\phi|x,t)\right|\le \frac{t}{2}\sup_{x\in \R^3}|\partial_x^{\alpha}\phi(x)|\\
\end{equation}
holds for any $t\in \R$.
\end{lem}

This lemma can be proved by the representation formula (\ref{u^0}) and (\ref{solop}) and a simple computation.

Now we are ready to prove Theorem \ref{lwp}.

\begin{proof}[Proof of Theorem \ref{lwp}]
We only consider the case $\gamma\in \left(-\frac{1}{2},3\right)$ with $\gamma\ne 2$, since the case of $\gamma=2$ can be treated in the almost similar manner.\\
(Existence) Let $M\in \R$ with $M>6N_{\gamma}(4)C_0r$, where $N_{\gamma}:\R_{\ge 0}\rightarrow \R_{\ge 0}$ is given by (\ref{wightf}) and $C_0$ is given by (\ref{linear est}). We take $T\in (0,R)$ sufficiently small such as
\begin{equation}
\label{3-33-2}
      T\le \left(\frac{2\pi}{3M^2C_1R^{3-\gamma}}\right)^{\frac{1}{2}},
\end{equation}
where $C_1$ is a positive constant given by (\ref{potential}). We introduce a complete metric space $X(T,R,M)$ in the Banach space $X(T,R)$ given by
\begin{equation}
\label{closedb}
   X(T,R,M)=X(T,M):=\left\{u\in X(T)\ :\ \|u\|_{X(T)}\le M\right\}
\end{equation}
with a metric
\[
   d_T(u_1,u_2):=\|u_1-u_2\|_{X(T)}.
\]
We prove that the nonlinear mapping $\Gamma$ defined by (\ref{nonlima}) is a contraction mapping from $X(T,M)$ into itself for a suitable $T$.
Fix $u\in X(T,M)$. By the condition (\ref{supportc}) on the support $u^0$, $R>1$ and Lemma \ref{lem:decay_est}, the estimate
\begin{align}
\label{3-34-1}
     \left\|u^0\right\|_{X(T)}\le 3N_{\gamma}(4)C_0r<\frac{M}{2}
\end{align}
holds. By Theorem \ref{lem:potential}, the estimate
\begin{equation}
\label{3-35}
    \left|(V_{\gamma}*u^2)(x,t)\right|\le C_1R^{3-\gamma}\|u\|_{X(T)}^2
\end{equation}
holds for any $(x,t)\in \R^3\times [0,T)$.
By the estimates (\ref{3-22-a}) and (\ref{3-35}), the inequalities
\begin{align*}
    &\tau_+(|x|,t)N_{\gamma}(\tau_-(|x|,t))\left|L((V_{\gamma}*u^2)u)(x,t)\right|\\
    &\le \int_0^t\tau_+(|x|,t)N_{\gamma}(\tau_-(|x|,t))\left|W\left(\frac{(V_{\gamma}*u^2)u(\cdot,s)}{(1+s)^2}\bigg|x,t-s\right)\right|ds\\
    &\le \frac{1}{4\pi}T^2\sup_{(x,t)\in \R^3\times [0,T)} \left|(V_{\gamma}*u^2)(x,t)\right|\tau_+(|x|,t)N_{\gamma}(\tau_-(|x|,t))|u(x,t)|\\
    &\le \frac{C_1R^{3-\gamma}}{4\pi}T^2\|u\|_{X(T)}^3\le \frac{C_1R^{3-\gamma}}{4\pi}T^2M^3
\end{align*}
hold for any $(x,t)\in \R^3\times[0,T)$, which with the estimate (\ref{3-33-2}) implies that the estimate
\begin{equation}
\label{3-36}
   \left\|L((V_{\gamma}*u^2)u)\right\|_{X(T)}\le \frac{C_1R^{3-\gamma}}{4\pi}T^2M^3\le \frac{1}{2}M
\end{equation}
holds. By the estimates (\ref{3-34-1}) and (\ref{3-36}), the estimates
\begin{align*}
    \left\|\Gamma[u]\right\|_{X(T)}&\le \left\|u^0\right\|_{X(T)}+\left\|L((V_{\gamma}*u^2)u)\right\|_{X(T)}\le \frac{1}{2}M+\frac{1}{2}M=M
\end{align*}
hold, which implies that the nonlinear mapping $\Gamma$ is well-defined on $X(T,M)$. Let $u,w\in X(T,M)$. By a direct computation, the identities
\begin{align}
   &\Gamma[u](x,t)-\Gamma[w](x,t)=L((V_{\gamma}*u^2)u)(x,t)-L\left((V_{\gamma}*w^2)w\right)(x,t)\notag\\
   &=L\left(\left(V_{\gamma}*u^2\right)(u-w)\right)(x,t)+L\left(V_{\gamma}*(uw)(u-w)\right)(x,t)\notag\\
   &+L\left((V_{\gamma}*(u-w))w^2\right)(x,t)\label{3-38-5}
\end{align}
hold.
Thus in the similar manner as the proof of the estimate (\ref{3-36}), the estimates
\begin{align}
   d_T(\Gamma[u],\Gamma[w])&\le \frac{3}{2}\cdot\frac{C_1R^{3-\gamma}}{4\pi}T^2\left(\|u\|_{X(T)}^2+\|w\|_{X(T)}^2\right)\|u-v\|_{X(T)}\notag\\
   &\le \frac{3C_1R^{3-\gamma}}{4\pi}T^2M^2\|u-v\|_{X(T)}
   \le \frac{1}{2}d_T(u,w)\label{3-38-3}
\end{align}
hold, which implies that the nonlinear mapping $\Gamma$ is a contraction mapping from $X(T,M)$ into itself. Thus by the fixed point theorem, there exists a unique solution $u\in X(T,M)$ to the problem (\ref{IE_u_depend}). The finite propagation property follows from the explicit formula (\ref{solop}).\\
(Uniquenss): On the contrary, we assume that there exists $t\in (0,T_1)$ such that the relation $\sup_{x\in \R^3}|u(t,x)-w(t,x)|>0$ holds. Then we can define $t_0:=\inf\{t\in [0,T_1)\ :\ \sup_{x\in \R^3}|u(x,t)-w(x,t)|> 0\}$. Since $u,w\in C(\R^3\times [0,T_1))$, the identity $\sup_{x\in \R^3}|u(x,t_0)-w(x,t_0)|=0$ holds. Here we introduce a norm $\|\cdot\|_{X([a,b))}$ given by
\[
   \|u\|_{X([a,b))}:=\sup_{(x,t)\in \R^3\times[a,b)}\tau_+(|x|,t)N_{\gamma}\left(\tau_-(|x|,t)\right)|u(x,t)|.
\]
Then we can find $\tau_0>0$ such that the estimates $t_0+\tau<T_1$ and
\begin{align*}
  \|u-w\|_{X([t_0,t_0+\tau))}&\le \frac{3}{2}\cdot\frac{C_1R^{3-\gamma}}{4\pi}\tau^2\left(\|u\|_{X(T_1)}^2+\|w\|_{X(T_1)}^2\right)\|u-w\|_{X(([t_0,t_0+\tau))}\\
  &\le \frac{1}{2}\|u-w\|_{X(([t_0,t_0+\tau))}
\end{align*}
hold, which implies that $u\equiv w$ on $\R^3\times [t_0,t_0+\tau)$. This contradicts the definition of $t_0$. Therefore $u\equiv v$ on $\R^3\times [0,T_1)$. \\
(Continuity of the flow map): Let $u\in X(T,M)$ and $w\in X(T,M)$ be the solutions to (\ref{IE_u_depend}) with the initial data $(v_0^u,v_1^u)\in B_r(C_0^2(\R^3)\times C_0^1(\R^3))$ and $(v_0^w,v_1^w)\in B_r(C_0^2(\R^3)\times C_0^1(\R^3))$ on $\R^3\times [0,T(r))$ respectively. In the same manner as the proof of the estimate (\ref{3-38-3}), the estimates
\begin{align*}
    &\|u-w\|_{X(T)}\\
    &\le 3N_{\gamma}(4)C_0\sup_{x\in\R^3}\bigg\{\sum_{|\alpha|\le 2}\left|\partial_x^{\alpha}\{v_0^u(x)-v_0^w(x)\}\right|+\sum_{|\beta|\le 1}\left|\partial_x^{\beta}\{v_1^u(x)-v_1^w(x)\}\right|\bigg\}\\
    &+\frac{3C_1R^{3-\gamma}}{4\pi}T^2M^2\|u-w\|_{X(T)}
\end{align*}
hold, which implies that the inequality
\begin{align*}
&\|u-w\|_{X(T)}\\
    &\le 6N_{\gamma}(4)C_0\sup_{x\in\R^3}\bigg\{\sum_{|\alpha|\le 2}\left|\partial_x^{\alpha}\{v_0^u(x)-v_0^w(x)\}\right|+\sum_{|\beta|\le 1}\left|\partial_x^{\beta}\{v_1^u(x)-v_1^w(x)\}\right|\bigg\}
\end{align*}
holds, which completes the proof of the theorem.
\end{proof}

\subsubsection{Proof of Theorem \ref{lower}}
\ \ Next we give a proof of Theorem \ref{lower}. Its argument is based on that for the proof of \cite[Proposition 5.3]{IO16}.

\begin{proof}[Proof of Theorem \ref{lower}]
  If $T(\varepsilon)=\I$, then the conclusion holds.
  Thus we may assume that $T(\varepsilon)<\I$. Let $u\in X(T(\varepsilon))$ be a solution to (\ref{IVP-2}) with the initial data ($\varepsilon v_0,\varepsilon v_1$). Then we can define a continuous function $H:(0,T(\varepsilon))\rightarrow \R_{\ge 0}$ given by $H(t):=\|u\|_{X(t)}$. For a sufficiently large $\mathcal{M}$, which will be chosen later, we set
  \EQQS{
    T^*:=\sup\{T\in[0,T(\e)]\ :\ H(T)\le \mathcal{M}\e\}.
  }
  By the local existence theorem (Theorem \ref{lwp}), we see that $T^*>0$.
  The definition of the number $T^*$ gives $T^*<T(\varepsilon)$. By Lemma \ref{lem:decay_est} and Proposition \ref{lm:apriori1}, the estimates
\begin{align}
  H(T^*)
  &\le 3N_{\gamma}(4)C_0\varepsilon r+C_2R^{5-\ga}\|u\|_{X(t)}^3 D_{\gamma}(T^*)\notag\\
  &\le \{3N_{\gamma}(4)C_0r\mathcal{M}^{-1}+C_2R^{5-\ga} D(T^*) (\mathcal{M}\e)^2\}\mathcal{M}\e\label{3-39-3}
\end{align}
  hold.
  We choose $\mathcal{M}>0$ such as $3N_{\gamma}(4)C_0r\mathcal{M}^{-1}\le 1/4$. On the contrary, we assume that the inequality
  \EQQS{
    R^{5-\ga}C_1 D(T^*) (\mathcal{M}\e)^2<\frac{1}{4},
  }
holds. Then the the estimate $H(T^*)\le (1/2)\mathcal{M}\e$ holds. Since $H(t)$ is continuous, there exists $\widetilde{T}\in(T^*,T(\varepsilon))$ such that $H(\widetilde{T})\le \mathcal{M}\e$, which violates the maximality of $T^*$. Thus we obtain
  \EQQS{
    C_2R^{5-\ga}(\mathcal{M}\e)^2 D(T^*)\ge \frac{1}{4}.
  }
  In the case of $\gamma\in \left(-\frac{1}{2},0\right)$, we choose $\e_0=\e_0(\gamma,R,r)$ such that
  \[
   \varepsilon_1:=2^{\frac{\gamma}{2}}\left\{4C_2R^{5-\gamma}\mathcal{M}^2(-\gamma)\right\}^{\frac{1}{2}}.
  \]
  Then for any $\varepsilon\in (0,\varepsilon_1]$, the estimate
  \[
    T^*\ge A\varepsilon^{\frac{2}{\gamma}}
  \]
  holds, where $A=A(\gamma,R,r)>0$ is a positive constant given by
  \[
     A:=\frac{1}{2}R\left\{4C_2R^{5-\gamma}\mathcal{M}^2(-\gamma)\right\}^{\frac{1}{\gamma}},
  \]
  which completes the proof of the theorem.
\end{proof}

\subsubsection{Proof of Theorem \ref{T.1-1}}
\ \ Next we give a proof of Theorem \ref{T.1-1}.

\begin{proof}[Proof of Theorem \ref{T.1-1}]
We only consider the case $\gamma\in \left(0,3\right)$ with $\gamma\ne 2$, since the case of $\gamma=2$ can be treated in the almost similar manner.\\
(Global existence): Let $\mathcal{M}>0$ satisfy $\mathcal{M}^{-1}3N_{\gamma}(4)C_0<\frac{1}{2}$, where $C_0$ is given by (\ref{linear est}). Let $\epsilon>0$ be sufficiently small such as $3C_2R^{5-\gamma}(\mathcal{M}\epsilon)^2<\frac{1}{2}$.
We introduce a complete metric space $X(\infty,R,\mathcal{M}\epsilon)$ given by (\ref{closedb}) in the Banach space $X(\infty,R)$. We prove that the nonlinear mapping $\Gamma$ defined by (\ref{nonlima}) is a contraction mapping from $X(\infty,R,\mathcal{M}\epsilon)$ into itself. Let $u\in X(\infty,R,\mathcal{M}\epsilon)$. In the same manner as the proof of the estimate (\ref{3-39-3}), the inequalities
\begin{align}
\label{3-41-2}
   \left\|\Gamma[u]\right\|_{X(\infty)}&\le 3N_{\gamma}(4)C_0\epsilon+C_2R^{5-\gamma}\|u\|_{X(\infty)}^3\\
   &\le\left\{\mathcal{M}^{-1}3N_{\gamma}(4)C_0+C_2R^{5-\gamma}(\mathcal{M}\epsilon)^2\right\}\mathcal{M}\epsilon\le \mathcal{M}\epsilon\notag
\end{align}
hold, which implies that the nonlinear mapping $\Gamma$ is well-defined on $X(\infty,R,\mathcal{M}\epsilon)$. Let $u,w\in X(\infty,R,\mathcal{M}\epsilon)$. By the estimate (\ref{3-38-5}) and Proposition \ref{lm:apriori1}, the estimates
\begin{align*}
    d_{\infty}(\Gamma[u],\Gamma[w])&\le \frac{3}{2}C_2R^{5-\gamma}\left(\|u\|_{X(\infty)}^2+\|w\|_{X(\infty)}^2\right)\|u-w\|_{X(\infty)}\\
    &\le 3C_2R^{5-\gamma}(\mathcal{M}\epsilon)^2\|u-w\|_{X(\infty)}\le \frac{1}{2}d_{\infty}(u,w)
\end{align*}
hold, which implies that the nonlinear mapping $\Gamma$ is a contraction mapping on $X(\infty,R,\mathcal{M}\epsilon)$. Thus by the fixed point theorem, there exists a unique solution $u\in X(\infty,R,\mathcal{M}\epsilon)$ to the problem (\ref{IE_u_depend}).
(Scattering): Let $u\in X(\infty,R,\mathcal{M}\epsilon)$ be the global solution to (\ref{IE_u_depend}). We introduce a function $u_+:\R^3\times\left[0,\infty\right)\rightarrow\R$ given by
 \[
    u^+(x,t):=u(x,t)-\int_t^{\infty}W\left(\frac{(V_{\gamma}*u^2)u(\cdot,s)}{(1+s)^2}\bigg|(x,t-s)\right)ds.
 \]
Then we see that the identity $\partial_t^2u^+-\Delta u^+=0$ holds on $\R^3\times[0,\infty)$. In the same manner as the proof of the estimate (\ref{3-41-2}), we can prove the identity
\begin{align*}
   \tau_+(|x|,t)N_{\gamma}(\tau_-(|x|,t))|u(x,t)-u^+(x,t)|=o(1)
\end{align*}
as $t\rightarrow\infty$, which implies (\ref{2-11-2}). Moreover, by the relation $v=(1+t)^{-1}u$, the estimate
\[
     (1+t)^{-1}\tau_+(|x|,t)N_{\gamma}(\tau_-(|x|,t))|v(x,t)|\le \mathcal{M}\epsilon
\]
holds, which implies (\ref{2-12-2}). This completes the proof of the theorem.

\end{proof}

\section{Proof of the blow-up result}
\label{blow-up}

\ \ In this section, we give a proof of the small data blow-up and the almost sharp upper estimate of the lifespan (Theorem \ref{T.1}).

The essential part for the proof is to refine the argument of the proof of Theorem 6.1 in \cite{H20} to obtain the almost sharp upper estimate of the lifespan.
\begin{proof}[Proof of Theorem \ref{T.1}]
Let $T\in (0,T(\varepsilon))$ and $u=u_{\varepsilon}\in C(\R^3\times[0,T))$ be a solution to the integral equation (\ref{IE_u_depend}) with the initial data $\varepsilon v_0\equiv 0$ and $\varepsilon v_1\ge 0$ satisfying $v_1\not\equiv 0$. We note that since the given function $(v_0,v_1)$ belongs to $C^2(\R^3)\times C^1(\R^3)$, we can prove that the solution $u$ belongs to $C^2(\R^3\times[0,T))$ and becomes a classical solution to (\ref{IVP-2}). We also remark that since the data $(v_0,v_1)$ satisfy $v_0\equiv0$ and $v_1\ge 0$ and $v_1\not\equiv 0$, by Theorem \ref{lwp}, the estimate $u(x,t)>0$ holds on $\R^3\times [0,T)$.

We introduce a function $F=F_{\varepsilon}:[0,T)\rightarrow \R_{>0}$ defined by
\[
F(t)=F_{\varepsilon}(t):=\int_{\R^3}u_{\varepsilon}(x,t)dx.
\]
For any $t\in [0,T)$, the value $|F(t)|$ is finite, since $\supp u(t)\subset\{x\in\R^3: |x|\le t+R\}$. For simplicity, we assume that $R=1$. Moreover we can prove that $F\in C^2([0,T(\varepsilon)))$. Since $\supp u(t)$ is compact in $\R^3$, by
integrating (\ref{IVP-2}) over $\R^3$, the identity
\begin{equation}
\label{id-4.1}
F''(t)=\frac{1}{(1+t)^2}\int_{\R^3}(V_\ga*u^2)(x,t)u(x,t)dx
\end{equation}
holds for any $t\in (0,T)$. Since $v_1=v_1(|x|)$ is a radially symmetric function, in the same manner as the proof of the estimate (6.11) in \cite{H20}, the estimate
\begin{equation}
\label{frame}
F''(t)\ge\frac{2^{-\gamma}F(t)}{(1+t)^{\gamma+2}}\int_{\R^3}u^2(x,t)dx
\end{equation}
holds on $[0,T)$.
Since the estimate $u(x,t)\ge u^0(x,t)$ holds for any $(x,t)\in\R^3\times[0,T)$ due to the positivity of the nonlinear term and
$\supp u^0(t)\subset\{x\in\R^3: t-1\le|x|\le t+1\}$, we have
\begin{equation}
\label{4-2}
\int_{t-1\le|x|\le t+1}u^0(x,t)dx\le \int_{t-1\le|x|\le t+1}u(x,t)dx\le \sqrt{2}(t+1)\|u(\cdot,t)\|_{L^2(\R^3)}.
\end{equation}
Next, we estimate the left hand-side of this inequality.
It follows from $u^0_{tt}-\Delta u^0=0$ that the identity
\[
\frac{d^2}{dt^2}\int_{\R^3}u^0(x,t)dx=0
\]
holds for any $t\ge 0$. By the assumption $v_0\equiv0$, the identities
\begin{equation}
\label{4-3}
\int_{\R^3}u^0(x,t)=\e t\int_{\R^3}v_1(x)dx=:\varepsilon tC_0
\end{equation}
hold for any $t\ge 0$, where $C_0=C_0(v_1)$ is a positive constant given by
\[
    C_0:=\int_{\R^3}v_1(x)dx>0.
\]
Thus by combining the estimates (\ref{4-2}) and (\ref{4-3}), the estimate
\[
\|u(\cdot,t)\|_{L^2(\R^3)}\ge \frac{\e C_0 t }{\sqrt{2}(1+t)}
\]
holds for any $t\in [0,T)$. Substituting this inequality to the estimate (\ref{frame}), the estimate
\begin{equation}
\label{ineq-2}
F''(t)\ge \e^2C_0^22^{-\gamma-1}\frac{t^2F(t)}{(1+t)^{\gamma+4}}\ge 0
\end{equation}
holds for any $t\in [0,T)$. We note that by the assumption on the data, the estimates $\d F'(0)=\varepsilon\int_{\R^3}v_1(x)dx=\varepsilon C_0>0$ hold. Thus by the estimate (\ref{id-4.1}), the inequality $F'(t)>0$ holds for any $t\in [0,T)$. Multiplying
$F'(t)$ to the both side of (\ref{ineq-2}), the inequalities
\[
\begin{array}{lll}
\d \left(\frac{F'^2(t)}{2}\right)'
&\ge&\d
\left( \e^2C_0^22^{-\gamma-1}\frac{t^2F^2(t)}{2(1+t)^{\gamma+4}}\right)'-
\e^2C_0^22^{-\gamma-2}F^2(t)\left(\frac{t^2}{(1+t)^{\gamma+4}}\right)'\\
&=&\d \left( \e^2C_0^22^{-\gamma-1}\frac{t^2F^2(t)}{2(1+t)^{\gamma+4}}\right)'\\
&&\qquad-\e^2C_0^22^{-\gamma-2}F^2(t)\{2t(t+1)^{-\gamma-4}
-t^2(\gamma+4)(1+t)^{-\gamma-5}\}
\end{array}
\]
hold for any $t\in [0,T)$. Here we set $t_{\gamma}:=\frac{2}{2+\gamma}$. We note that by the estimate $\gamma>-2$, the estimate $t_{\gamma}>0$ and $t_{\gamma}$ is independent of $\varepsilon$. We may assume that $T(\varepsilon)>t_{\gamma}$. Let $T\in (t_{\gamma},T(\varepsilon))$ and $t\in (t_{\gamma},T)$. Then the estimates
\[
\begin{array}{lll}
2t(t+1)^{-\gamma-4}-t^2(\gamma+4)(1+t)^{-\gamma-5}
=t(t+1)^{-\gamma-5}\{2-(\gamma+2)t\}\le0.
\end{array}
\]
Thus for any $t\in (t_{\gamma},T)$, the estimate
\[
\left(\frac{F'^2(t)}{2}\right)'
\ge
\left( \e^2C_0^22^{-\gamma-1}\frac{t^2F^2(t)}{2(1+t)^{\gamma+4}}\right)'
\]
holds. Integrating this inequality from $t_{\gamma}$ to $t$, the estimate
\begin{equation}
\label{positive-3}
F'^2(t)\ge \e^2C_0^22^{-\gamma-1}\frac{t^2F^2(t)}{(1+t)^{\gamma+4}}
+F'^2(t_{\gamma})
-\frac{\e^2C_0^22^{-\gamma-1}t_{\gamma}^2F^2(t_{\gamma})}{(1+t_{\gamma})^{\gamma+4}}
\end{equation}
holds for any $t\in (t_{\gamma},T)$.

Next we show that there exists a positive constant $\varepsilon_4=\varepsilon_4(\gamma,R)>0$ such that for any $\varepsilon\in (0,\varepsilon_4]$, the estimate
\begin{equation}
\label{positive-1}
\left\{F_{\varepsilon}'(t_{\gamma})\right\}^2
-\frac{\e^2C_0^22^{-\gamma-1}t_{\gamma}^2\left\{F_{\varepsilon}(t_{\gamma})\right\}^2}{(1+t_{\gamma})^{\gamma+4}}\ge0
\end{equation}
holds. Here we set
\[
\e_4:=\mathcal{M}^{-1}2^{(\gamma+1)/2}(1+t_{\gamma})^{(\gamma+9)/2}t_{\gamma}^{-1},
\]
where $\mathcal{M}$ is a positive constant independent of $\varepsilon$. By the estimate $\gamma\in \left(-\frac{1}{2},0\right)$, in the same manner as the proof of the estimate (\ref{3-39-3}), we can prove that there exists a positive constant $\varepsilon_5>0$ such that for any $\varepsilon \in [0,\varepsilon_5]$, the estimate
\begin{equation}
\label{4-7}
|u(x,t)|\le \mathcal{M}\e(t+|x|+2)^{-1}(t-|x|+2)^{-\gamma-1}
\end{equation}
holds for any $t\in [0,T)$ and $|x|\le t+1$. Set $\varepsilon_6:=\min(\varepsilon_4,\varepsilon_5)$. Let $\varepsilon\in (0,\varepsilon_6]$. Then by H\"older's inequality and the estimate (\ref{4-7}), the estimates
\begin{equation}
\label{positive-2}
\begin{array}{llll}
\d F(t_{\gamma})
&\le&\d
  C_1(1+t_{\gamma})^{3/2}\left(\int_{|x|\le 1+t_{\gamma}}u^2(x,t_{\gamma})dx\right)^{1/2}\\
&\le&\d C_1\mathcal{M}\e(1+t_{\gamma})^{3/2}\left(\int_{0}^{1+t_{\gamma}}(t_{\gamma}+r+2)^{-2}(t_{\gamma}-r+2)^{-2\gamma-2}dr\right)^{1/2} \\
&\le&\d C_1\mathcal{M}\e (1+t_{\gamma})
\end{array}
\end{equation}
hold, where $C_1$ is a positive constant independent of $\varepsilon$. Since the estimate $F'(t_{\gamma})\ge \e C_0 $ holds, by the estimate (\ref{positive-2}), the inequalities
\[
F'^2(t_{\gamma})
-\frac{\e^2C_0^22^{-\gamma-1}t_{\gamma}^2F^2(t_{\gamma})}{(1+t_{\gamma})^{\gamma+4}}
\ge \e^2C_0^2
-\e^4\mathcal{M}^2C_1^2C_0^22^{-\gamma-1}t_{\gamma}^2(1+t_{\gamma})^{-\gamma-2}\ge0
\]
hold for $0<\e\le\e_6$.
By the estimate (\ref{positive-3}), the inequality
\[
F'(t)\ge \e C_02^{-(\gamma+1)/2}\frac{tF(t)}{(1+t)^{(\gamma+4)/2}}
\]
holds for $t\in (t_{\gamma},T)$ and $0<\e\le\e_6$.
Set $t_0:=\max\{1,t_{\gamma}\}$. We note that $t_0$ is independent of $\varepsilon$. Then we may assume that $T(\varepsilon)>t_0$. Let $T\in (t_0,T(\varepsilon))$. Then the estimate
\[
F'(t)\ge\e C_02^{-(2\gamma+5)/2}t^{-1-\gamma/2}F(t)
\]
holds for any $t\in [t_0,T)$. We solve this ordinary differential inequality.
Since the estimate $F(t)>0$ holds, the inequality gives us
\[
(\log F(t))'\ge\e C_02^{-(2\gamma+5)/2}t^{-1-\gamma/2}
\]
for any $t\in [t_0,T)$. Integrating this inequality over $[t_0,t]$, the inequality
\[
\log F(t)-\log F(t_0)\ge\e C_02^{-(2\gamma+5)/2}(-2/\gamma)\{t^{-\gamma/2}-t_0^{-\gamma/2}\}
\]
holds for any $t\in [t_0,T)$. Let $t_1:=(2t_0^{-\gamma/2})^{-2/\gamma}$ and $t_2:=\max(t_0,t_1)$. Then we may assume that $T(\varepsilon)>t_2$. Let $T\in (t_2,T(\varepsilon))$. The estimate
\[
\log F(t)\ge \e C_02^{-(2\gamma+5)/2}(-\gamma^{-1})t^{-\gamma/2}+\log F(t_0)
\]
holds for any $t\in (t_2,T)$, which implies that the inequalities
\begin{equation}
\label{exp-est}
\begin{array}{lll}
F(t)&\ge& F(t_0)\exp\{ \e C_02^{-(2\gamma+5)/2}(-\gamma^{-1})t^{-\gamma/2}\}\\
&\ge&
\e C_0t_0 \exp\left( \e C_2 t^{-\gamma/2}\right)
\end{array}
\end{equation}
for any $t\in (t_2,T)$, where $C_2:=C_02^{-(2\gamma+5)/2}(-\gamma)^{-1}>0$ is a positive constant independent of $\varepsilon$.
We note that for any $\vartheta>0$, the identity $\exp \vartheta=\sum_{j=0}^{\infty}\frac{1}{j!}\vartheta^j$ holds. This convergence is uniform with respect to $\vartheta>0$. Then the estimate
\begin{equation}
\label{4-10}
    \exp\left(\varepsilon C_2t^{-\gamma/2}\right)\ge \frac{1}{j!}(\varepsilon C_2t^{-\gamma/2})^j
\end{equation}
holds. By combining the estimates (\ref{exp-est}) and (\ref{4-10}), the inequality
\[
    F(t)\ge \varepsilon^{1+j}C_0t_0\left(\frac{1}{j!}C_2^j\right)t^{-\frac{\gamma}{2}j}
\]
holds for any $t\in (t_2,T)$.

We turn back to the inequality (\ref{frame}). By the finite propagation speed and H\"older's inequality, the estimates
\begin{equation}
\label{frame-2}
F''(t)\ge 2^{-\gamma-2}\cdot3\pi^{-1}(1+t)^{-(\gamma+5)}F^3(t)
\end{equation}
hold for $t\in [0,T)$.
Let $\delta>0$. Then there exists $J_1(\delta)>0$ such that for any $j\in\N$ 
with $j\ge J_0(\delta)$, the estimate $\frac{2(j+1)}{\gamma(j+1)+3}>\frac{2}{\gamma}-\delta$ holds.
Set $J_2=\max\{J_1(\delta),[-3/\gamma-1]+1\}$, where $[\cdot]$ is the integral part. 
For $j= J_2$, we apply Lemma 2.1 in \cite{T15} (Improved Kato's lemma) with $R=1$,
\[
   p:=3,\ q:=\gamma+5>0, a:=-\frac{\gamma}{2}j>0,\ A:=\varepsilon^{1+j}C_0t_0\left(\frac{1}{j!}C_2^j\right),\ B:=2^{-\gamma-2}\cdot3\pi^{-1}
\]
and
\[
   M:=\frac{p-1}{2}a-\frac{q}{2}+1=-\frac{\gamma(j+1)+3}{2}>0.
\]

We set
\[
   T_0:=D_0A^{-\frac{p-1}{2M}}=D_0\varepsilon^{\frac{2(j+1)}{\gamma(j+1)+3}}\left(C_0t_0\right)^{\frac{2}{\gamma(j+1)+3}}\left(j!\right)^{\frac{2}{-\gamma(j+1)-3}}C_0^{\frac{2j}{\gamma(j+1)+3}},
\]
where $D_0=D_0(p,a,q,B)>0$ is a positive constant independent of $\varepsilon$, which appears in Lemma 2.1 in \cite{T15} (in his paper \cite{T15}, $D_0$ is written by $C_0$). 
Then there exists $\varepsilon_7(j)\in (0,1]$ such that for any $\varepsilon\in (0,\varepsilon_7(j)]$, the estimates
\[
    T_0\ge \max\left(\frac{F(0)}{F'(0)},R\right)=R
\]
hold, which implies that $T_1=T_0$ in Lemma 2.1 in \cite{T15}. Let $\varepsilon\in (0,\varepsilon_7(j))$. Thus by the conclusion of Lemma 2.1 in \cite{T15}, the estimates
\[
\begin{array}{lll}
   T<2^{\frac{2}{M}}T_1
&=&2^{-\frac{4}{\gamma(j+1)+3}}D_0\varepsilon^{\frac{2(j+1)}{\gamma(j+1)+3}}\left(C_0t_0\right)^{\frac{2}{\gamma(j+1)+3}}\left(j!\right)^{\frac{2}{-\gamma(j+1)-3}}
C_0^{\frac{2j}{\gamma(j+1)+3}}\\
&\le&2^{-\frac{4}{\gamma(j+1)+3}}D_0\varepsilon^{\frac{2}{\gamma}-\delta}\left(C_0t_0\right)^{\frac{2}{\gamma(j+1)+3}}\left(j!\right)^{\frac{2}{-\gamma(j+1)-3}}C_0^{\frac{2j}{\gamma(j+1)+3}}
\end{array}
\]
hold, which completes the proof of the theorem.

\end{proof}

\section*{Acknowledgement}
\par
The first author is supported by JST CREST Grant Number JPMJCR1913, Japan.
The first and third authors have been partially supported by
the Grant-in-Aid for Scientific Research (B) (No.18H01132) and Young Scientists Research (No.19K14581), Japan Society for the Promotion of Science.
The second author has been supported by RIKEN Junior Research Associate Program.
The third author has been partially supported
by the Grant-in-Aid for Scientific Research (B) (No.19H01795), Japan Society for the Promotion of Science.


\bibliographystyle{plain}

\begin{thebibliography}{20}
\bibitem{AKT00}{R. Agemi, Y. Kurokawa and H. Takamura},
{\it Critical curve for $p$-$q$ systems of nonlinear wave equations in three space dimensions},
J. Differential Equations {\bf 167} (2000), no. 1, 87--133.

\bibitem{D15}{M. D'Abbicco},
{\it The threshold of effective damping for semilinear wave equations},
Math. Methods Appl. Sci. {\bf 38} (2015), no. 6, 1032--1045.

\bibitem{DLR15}{M. D'Abbicco, S. Lucente and M. Reissig},
{\it A shift in the Strauss exponent for semilinear wave equations with a not effective damping},
J. Differential Equations {\bf 259} (2015), no. 10, 5040--5073.

\bibitem{DL15}{M. D'Abbicco and S. Lucente},
{\it NLWE with a special scale invariant
damping in odd space dimension},
Discrete Contin. Dyn. Syst. 2015,
Dynamical systems, differential equations and applications.
10th AIMS Conference. Suppl., 312--319.


\bibitem{H20}{K. Hidano},
{\it Small data scattering and blow-up for a wave equation with a cubic convolution}, Funkcial. Ekvac. {\bf 43} (2000), 559--588.

\bibitem{IO16}{M. Ikeda and T. Ogawa},
{\it Lifespan of solutions to the damped wave equation with a critical nonlinearity},
J. Differential Equation, {\bf 261} (2016), 1880--1903.

\bibitem{IkeSoba}{M. Ikeda and M. Sobajima},
{\it Life-span of solutions to semilinear wave equation
with time-dependent critical damping
for specially localized initial data},
Math. Ann. {\bf 372} (2018), no. 3-4, 1017--1040.

\bibitem{ISW19}{M. Ikeda, M. Sobajima and K. Wakasa},
{\it Test function method for blow-up phenomena of semilinear wave equations and their weakly coupled systems}, J. Differential Equations, {\bf 16} (2019), 495--517.

\bibitem{ISWk19}{M. Ikeda, M. Sobajima and Y. Wakasugi},
{\it Sharp lifespan estimates of blowup solutions to semilinear wave equations with time-dependent effective damping},
J. Hyperbolic Differ. Equ. {\bf 11} (2019), 795--819.

\bibitem{ITW}{M. Ikeda, T. Tanaka and K. Wakasa},
{\it Small data blow-up for the wave equation with a time-dependent scale invariant damping and a cubic convolution for slowly decaying initial data}, arXiv:2001.07985.

\bibitem{KT}{P.Karageorgis and K.Tsutaya},
{\it Existence and blowup for a Hartree-type wave equation}, in preparation.


%



\bibitem{LTW17}{N.-A. Lai, H. Takamura and K. Wakasa},
{\it Blow-up for semilinear wave equations
with the scale invariant damping and super-Fujita exponent},
J. Differential Equations {\bf 263} (2017), no. 9, 5377--5394.

\bibitem{J55}{F. John},
\lq\lq Plane Waves and Spherical Means,
Applied to Partial Differential Equations",
Interscience Publishers, Inc., New York, 1955.


\bibitem{J79}{F. John},
{\it Blow-up of solutions of nonlinear wave equations in three space dimensions},
Manuscripta Math., {\bf 28}(1979), 235--268.

\bibitem{KS}{M. Kato and M. Sakuraba},
{\it Global existence and blow-up for semilinear damped wave equations
in three space dimensions},
Nonlinear Anal. {\bf 182} (2019), 209--225.

\bibitem{KTW} M. Kato, H. Takamura and K. Wakasa,
{\it The lifespan of solutions of semilinear wave equations with the scale-invariant damping in one space dimension},
arXiv:1810.03780.

\bibitem{K85}{S. Klainerman},
{\it Uniform decay estimate and the Lorentz invariance of the classical wave equations},
Comm. Pure Appl. Math., {\bf 40} (1985), 321--332.




\bibitem{K04}{H. Kubo},
{\it On Point-Wise Decay Estimates for the Wave Equation and Their Applications},
Dispersive nonlinear problems in mathematical physics, 123--148, Quad. Mat., 15, Dept.
Math., Seconda Univ. Napoli, Caserta, 2004.

\bibitem{Lai}{N.-A. Lai},
{\it Weighted $L^2 - L^2$ estimate for wave equation and its applications},
arXiv:1807.05109.

\bibitem{MS82}{G. P. Menzala and W. A. Strauss},
{\it On a wave equation with a cubic convolution}, J. Differential Equations {\bf 43} (1982), 93--105.



\bibitem{S85}{J. Schaeffer},
{\it The equation $u_{tt}-\Delta u=|u|^{p}$ for the critical value of $p$}. Proc. Roy. Soc.
Edinburgh {\bf 101A} (1985), 31--44.


\bibitem{T15}{H.Takamura},
{\it  Improved Kato's lemma on ordinary differential inequality and its application to semilinear wave equations}, Nonlinear Anal., {\bf 125}, (2015) 227--240.









\bibitem{T93}{K. Tsutaya},
{\it Global existence theorem for semilinear wave equations with non-compact data in two space dimensions},
J. Differential Equations {\bf 104} (1993) 332--360.

\bibitem{T94}{K. Tsutaya},
{\it Global existence and the lifespan of solutions of semilinear wave equations with data of non compact support in three space dimensions}, Funkcial. Ekvac. {\bf 37} (1994) 1--18.


\bibitem{T03}{K. Tsutaya},
{\it Global existence and blow up for a wave equation with a potential and a cubic convolution}, Nonlinear Analysis and Applications: to V. Lakshmikantham on his 80th
Birthday. Vol. 1, 2, 913--937, Kluwer Acad. Publ., Dordrecht, 2003.

\bibitem{T14}{K. Tsutaya},
{\it Weighted estimates for a convolution appearing in the wave equation of Hartree type},
J. Math. Anal. Appl. {\bf 411} (2014), 719--731.

\bibitem{TL1}{Z. Tu and J. Lin},
{\it A note on the blowup of scale invariant damping wave equation
with sub-Strauss exponent},
arXiv:1709.00866.

\bibitem{TL2}{Z. Tu and J. Lin},
{\it Life-span of semilinear wave equations
with scale-invariant damping: critical Strauss exponent case},
Differential Integral Equations {\bf 32} (2019), no. 5--6, 249-264.

\bibitem{YZ06}{B. Yordanov and Q.S. Zhang},
{\it Finite time blow up for critical wave equations in high dimensions}, J. Funct. Anal., {\bf 231} (2006), 361--374.

\bibitem{Wa16}{K. Wakasa},
{\it The lifespan of solutions to semilinear damped wave equations in one space dimension},
Comm. Pure Appl. Anal. {\bf 15} (2016), no. 4, 1265--1283.


\bibitem{Wakasugi14}{Y. Wakasugi},
{\it Critical exponent for the semilinear wave equation with scale invariant damping},
Fourier analysis, 375-390, Trends Math., Birkh\"{a}user/Springer, Cham, 2014.

\bibitem{Wi06}{J. Wirth},
{\it Wave equations with time-dependent dissipation. I. Non-effective dissipation}, J. Differential Equations, {\bf 222} (2006), 487--514.

\bibitem{Z07}{Y. Zhou},
{\it Blow up of solutions to the Cauchy problem for nonlinear wave equations}, Chin. Ann. Math. Ser. B, {\bf 14} (2007), 205--212.

\end{thebibliography}

\end{document}